\documentclass[a4paper, 12pt]{amsart}
\usepackage{fullpage}
\usepackage{PaperStyle}
\usepackage{youngtab}
\usepackage[smalltableaux, centertableaux]{ytableau}

\newcommand{\D}{\connection}
\renewcommand{\d}{\mathrm{d}}

\newcommand{\bM}{\bar{M}}
\newcommand{\bg}{\bar{g}}
\newcommand{\bE}{\bar{E}}

\begin{document}

\title
{The range of a connection and a Calabi operator for Lorentzian locally
symmetric spaces}

\thanks{This work was supported by
the Australian Research
Council (Discovery Program DP190102360).}

\author[E.F.~Costanza]{Federico Costanza}
\address{\hskip-\parindent
Center for Theoretical Physics\\
Polish Academy of Sciences\\
Al. Lotnik\'{o}w 32/46, {02-668} Warsaw\\Poland}
\email{efcostanza@gmail.com}
\author[M.G.~Eastwood]{Michael Eastwood}
\address{\hskip-\parindent
School of Computer and Mathematical Sciences\\
University of Adelaide\\
SA 5005\\
Australia}
\email{meastwoo@gmail.com}
\author[T.~Leistner]{Thomas Leistner}
\address{\hskip-\parindent
School of Computer and Mathematical Sciences\\
University of Adelaide\\
SA 5005\\
Australia}
\email{thomas.leistner@adelaide.edu.au}
\author[B.B.~McMillan]{Benjamin McMillan}
\address{\hskip-\parindent
Center for Complex Geometry, Institute for Basic Science,
55, Expo-ro, Yuseong-gu, Daejeon, Korea, 34126
} \email{mcmillanbb@gmail.com}

\begin{abstract}
For semi-Rieman\-nian manifolds of constant sectional curvature, the Calabi
operator is a second order linear differential operator that provides local
integrability conditions for the range of the Killing operator.  In this
article, extending earlier results in the Rieman\-nian setting, we identify the
Lorentzian locally symmetric spaces on which the Calabi operator is sufficient
to identify the range of the Killing operator.  Specifically, this is always
the case for indecomposable spaces and we identify precisely those products for
which it fails.  Our method is quite general in that we firstly develop
criteria to be in the range of a connection, viewed as a linear differential
operator.  Then we ascertain how these criteria apply in the case of what we
call the Killing connection.
\end{abstract}

\maketitle

\section{Introduction}

In a previous article~\cite{CostanzaEastwoodLeistner21}, we found local integrability conditions
for the range of the Killing operator
\begin{equation}\label{killing_operator}
X_a\stackrel{{\mathcal{K}}}{\longmapsto}\nabla_{(a}X_{b)}
\end{equation}
on an arbitrary irreducible Rieman\-nian symmetric space as the kernel of a
second order linear differential operator.
Here, the $1$-form $X_a$ is, equivalently, a vector field $X^a$ with its index
lowered $X_a=g_{ab}X^b$ using the metric $g_{ab}$ in the usual way and
$\nabla_a$ is the Levi-Civita connection for~$g_{ab}$.  More generally, we are
employing Penrose's {\em abstract index notation} ~\cite{penroserindler1} for
tensors.  In particular, round/square brackets mean to take the symmetric/skew part of a
tensor and a repeated index denotes the invariant pairing between vectors and
co-vectors (echoing the {\em Einstein summation convention}).
Crucial for our results in~\cite{CostanzaEastwoodLeistner21} is the
{\em Calabi operator}, a linear second order
differential operator ${\mathcal{C}}$ acting on symmetric bilinear forms as
\begin{equation}\label{calabi}h_{ab}\mapsto
\nabla_{(a}\nabla_{c)}h_{bd}-\nabla_{(b}\nabla_{c)}h_{ad}
-\nabla_{(a}\nabla_{d)}h_{bc}
+\nabla_{(b}\nabla_{d)}h_{ac}-R_{ab}{}^e{}_{[c}h_{d]e}
-R_{cd}{}^e{}_{[a}h_{b]e},
\end{equation}
where $R_{ab}{}^c{}_d$ is the Riemann curvature tensor characterised by
$$(\nabla_a\nabla_b-\nabla_b\nabla_a)X^c=R_{ab}{}^c{}_dX^d,\quad \text{for any vector field~$X^c$.}$$
In~\cite{Calabi61}, Calabi showed that if the
Riemannian metric $g_{ab}$ is of constant sectional curvature,
then the range of $\mathcal K$ is equal to the kernel of $\mathcal C$. In \cite{CostanzaEastwoodLeistner21} we generalise this result appropriately to  Riemannian locally symmetric spaces.

This article is a follow-up to~\cite{CostanzaEastwoodLeistner21} with the following two outcomes.
Firstly, we place the machinery of \cite{CostanzaEastwoodLeistner21} in a more general setting so
that, in principle, one can study {\em any\/} overdetermined linear
differential operator in place of the Killing operator.  Secondly, we use this
machinery to extend the results of \cite{CostanzaEastwoodLeistner21} to the Lorentzian case (see
Theorem~\ref{calabi-theo}).

The prototypical overdetermined linear differential operator is a connection
\begin{equation}\label{D}D:E\to\Wedge^1\otimes E\end{equation}
on a smooth vector bundle~$E$ over a smooth manifold $M$ and the general case
can often be captured by a suitable connection.  We shall come back to this
{\em prolongation procedure\/} soon but, firstly, let us make some general
remarks concerning the range of (\ref{D}).  It is well-known (see
e.g.~\cite{MilnorStasheff74}), that any connection on $E$ gives rise to a
natural sequence of operators
\begin{equation}\label{coupled_de_Rham}
E\xrightarrow{\,D\,}\Wedge^1\otimes E\xrightarrow{\,\dconn\,} \Wedge^2\otimes E
\xrightarrow{\,\dconn\,}\cdots\xrightarrow{\,\dconn\,} \Wedge^p\otimes E
\xrightarrow{\,\dconn\,}\Wedge^{p+1}\otimes E\xrightarrow{\,\dconn\,}
\cdots\end{equation}
and that the composition $E\to\Wedge^1\otimes E\to\Wedge^2\otimes E$ is
actually a homomorphism, namely the {\em curvature\/} of $D$, which we shall
denote by $\kappa\in\Hom(E,\Wedge^2\otimes E)$.  It is also well-known that, if
the curvature vanishes, then (\ref{coupled_de_Rham}) is a locally exact
complex.  In fact, in this so-called {\em flat\/} case, we can identify the
range of $D:E\to\Wedge^1\otimes E$ on any open simply-connected subset
$U\subseteq M$:
$$\{D\phi\in\Gamma(U,\Wedge^1\otimes E)\mbox{ for }\phi\in\Gamma(U,E)\}
=\{\psi\in\Gamma(U,\Wedge^1\otimes E)\mbox{ s.t.\ }\dconn \psi=0\}.$$
More generally, if $\Gamma(\Wedge^1\otimes E)\ni\psi=D\phi$, then
$\dconn\psi=\kappa\phi$ so, in particular, it follows that $\dconn\psi$ must be in the
range of $\kappa:E\to\Wedge^2\otimes E$.  This necessary differential condition
on $\psi$ is sometimes locally sufficient.  In other words, assuming that the
range of the homomorphism $\kappa$ is actually a subbundle of $\Wedge^2\otimes
E$ (as we shall suppose throughout this article) the induced {\em complex\/} of
differential operators
\begin{equation}\label{curvature_reduced_complex}
E\xrightarrow{\,D\,}\Wedge^1\otimes E\to\frac{\Wedge^2\otimes E}{\kappa(E)}
\end{equation}
is sometimes locally exact.  Unfortunately, this is not always the case, and in
the first part of this article, we develop criteria for the local exactness
of~(\ref{curvature_reduced_complex}), and present alternative integrability
conditions on the range of $D$ in case that these criteria fail.  In
particular, we formulate an invariantly defined {\em second order\/}
differential operator
$$\begin{array}{ccc}{\mathcal{D}}:\Wedge^1\otimes E
&\longrightarrow&\Delta^2\otimes E\\
&\searrow&\downarrow\\[-4pt]
&\raisebox{7pt}{\makebox[0pt][r]{\scriptsize$\dconn$}}&\Wedge^2\otimes E
\end{array}$$
for an appropriate bundle $\Delta^2$ and so that the composition
$E\xrightarrow{\,D\,}\Wedge^1\otimes E
\xrightarrow{\,{\mathcal{D}}\,} \Delta^2\otimes E$ is
still a {\em homomorphism\/} of vector bundles.  We call this composition the
{\em augmented curvature\/} of $D$ and use it to strengthen the
criterion for the range of $D$ in case that (\ref{curvature_reduced_complex})
fails to be locally exact.

As already indicated, these general considerations regarding the range of a
connection can be applied to determine the local range of other overdetermined
operators.  Implicitly, this was our strategy in \cite{CostanzaEastwoodLeistner21} for the Killing
operator~(\ref{killing_operator}).  Here, we present it more explicitly as
follows.  The particular connection that we employed in~\cite{CostanzaEastwoodLeistner21} was used
by Kostant~\cite{Kostant55} in a similar context.  We call it the {\em Killing
connection\/}, defined on any semi-Rieman\-nian manifold as follows.  Consider
the vector bundle $E\equiv\Wedge^1\oplus\Wedge^2$ and define a connection on
this bundle by
\begin{equation}\label{killing_connection}
E=\begin{array}{c}\Wedge^1\\[-3pt] \oplus\\[-1pt]
\Wedge^2\end{array}\ni\left[\!\begin{array}{c}\sigma_c\\ \mu_{cd}
\end{array}\!\right]\stackrel{D_b}{\longmapsto}
\left[\!\begin{array}{c}\nabla_b\sigma_c-\mu_{bc}\\
\nabla_b\mu_{cd}-R_{cd}{}^e{}_b\sigma_e
\end{array}\!\right]\in\Wedge^1\otimes E.\end{equation}
The curvature of this {\em Killing connection} will be crucial in our approach. By applying $D_a$ to~(\ref{killing_connection}),
\[D_aD_b\left[\!\begin{array}{c}\sigma_c\\ \mu_{cd}
\end{array}\!\right]=
\left[\!\begin{array}{c}\nabla_a\nabla_b\sigma_c-\nabla_a\mu_{bc} -\nabla_b\mu_{ac}+R_{bc}{}^e{}_a\sigma_e\\
\nabla_a\nabla_b\mu_{cd}-\nabla_aR_{cd}{}^e{}_b\sigma_e -R_{cd}{}^e{}_b\nabla_a\sigma_e-R_{cd}{}^e{}_a\left( \nabla_b\sigma_e- \mu_{be}\right)
\end{array}\!\right]
\]
one obtains the curvature $\kappa$ of $D_a$ as
\begin{equation}\label{Killing_curvature}
(D_aD_b-D_bD_a)\left[\!\begin{array}{c}\sigma_c\\ \mu_{cd}
\end{array}\!\right]=\left[\!\begin{array}{c}0\\
2R_{ab}{}^e{}_{[c}\mu_{d]e}+2R_{cd}{}^e{}_{[a}\mu_{b]e}
-(\nabla^eR_{abcd})\sigma_e
\end{array}\!\right],
\end{equation}
using the Bianchi symmetry
$R_{[abc]d}=0$, and Bianchi identity $\nabla_{[a}R_{bc]de}=0$.
Moreover, the sequence \eqref{coupled_de_Rham} fits profitably in the following commutative diagram,
\begin{equation}\label{shiny_new_diagram}\begin{array}{cccccccccccc}
0 & \to & \Wedge^2&\to&
\begin{array}{c}\Wedge^2\\[-3pt]\oplus\\[-2pt]
\Wedge^2\otimes\Wedge^1\end{array}&\to&
\begin{array}{c}\Wedge^2\otimes\Wedge^1\\[-3pt]\oplus\\[-2pt]
\Wedge^3\otimes\Wedge^1\end{array}&\to&
\begin{array}{c}\Wedge^3\otimes\Wedge^1\\[-3pt]\oplus\\[-2pt]
\Wedge^4\otimes\Wedge^1\end{array}&\to&\cdots\\
&& \uparrow&&\uparrow&&\uparrow&&\uparrow\\
&& E&\stackrel{D}{\longrightarrow}&\Wedge^1\otimes E&
\stackrel{\dconn}{\longrightarrow}&\Wedge^2\otimes E&
\stackrel{\dconn}{\longrightarrow}&\Wedge^3\otimes E&
\stackrel{\dconn}{\longrightarrow}&\cdots\\
&& \uparrow&&\uparrow&&\uparrow&&\uparrow\\
&& \Wedge^1&\stackrel{\mathcal{K}}{\longrightarrow}&
\bigodot^2\!\Wedge^1&\stackrel{\mathcal{C}/2}{\longrightarrow}&
\begin{picture}(12,6)
\put(0,-6){\line(1,0){12}}
\put(0,0){\line(1,0){12}}
\put(0,6){\line(1,0){12}}
\put(0,-6){\line(0,1){12}}
\put(6,-6){\line(0,1){12}}
\put(12,-6){\line(0,1){12}}
\end{picture}&\stackrel{\mathcal{B}}{\longrightarrow}&
\begin{picture}(12,6)
\put(0,-12){\line(1,0){6}}
\put(0,-6){\line(1,0){12}}
\put(0,0){\line(1,0){12}}
\put(0,6){\line(1,0){12}}
\put(0,-12){\line(0,1){18}}
\put(6,-12){\line(0,1){18}}
\put(12,-6){\line(0,1){12}}
\end{picture}&\longrightarrow&\cdots,
\end{array}\end{equation}

\medskip\noindent
where $\bigodot$ denotes symmetric tensor product and
\begin{itemize}
\item \begin{picture}(12,12)
\put(0,0){\line(1,0){12}}
\put(0,6){\line(1,0){12}}
\put(0,12){\line(1,0){12}}
\put(0,0){\line(0,1){12}}
\put(6,0){\line(0,1){12}}
\put(12,0){\line(0,1){12}}
\end{picture} denotes the bundle of tensors $\mu_{bcde}=\mu_{[bc][de]}$
such that $\mu_{[bcd]e}=0$,
\item \begin{picture}(12,12)
\put(0,-6){\line(1,0){6}}
\put(0,0){\line(1,0){12}}
\put(0,6){\line(1,0){12}}
\put(0,12){\line(1,0){12}}
\put(0,-6){\line(0,1){18}}
\put(6,-6){\line(0,1){18}}
\put(12,0){\line(0,1){12}}
\end{picture} denotes the bundle of tensors $\mu_{abcde}=\mu_{[abc][de]}$
such that $\mu_{[abcd]e}=0$,
\item \rule{0pt}{14pt}the operator
${\mathcal{C}}:\bigodot^2\!\Wedge^1\to
\begin{picture}(12,12)(0,2)
\put(0,0){\line(1,0){12}}
\put(0,6){\line(1,0){12}}
\put(0,12){\line(1,0){12}}
\put(0,0){\line(0,1){12}}
\put(6,0){\line(0,1){12}}
\put(12,0){\line(0,1){12}}
\end{picture}$ is given by formula~(\ref{calabi}),
\vspace{1mm}
\item and the operator ${\mathcal{B}}$ is given by
$\mu_{bcde}\mapsto\nabla_{[a}\mu_{bc]de}$.
\end{itemize}
Explicit formul{\ae} for the yet undescribed maps of Diagram \eqref{shiny_new_diagram} are given in \S\ref{kill_and_more_kill}.
  It need only be said here that they are natural, and that once defined it is not difficult to show that the top row is exact, while each column is a short exact sequence, and that the diagram commutes.

From this, it is a formal calculation to demonstrate 
the crucial fact that if the curvature \( \kappa \) of the Killing connection has constant rank, then
 the range of ${\mathcal{C}}\circ{\mathcal{K}}$ defines a subbundle of
$\,\begin{picture}(12,12)(0,2)
\put(0,0){\line(1,0){12}}
\put(0,6){\line(1,0){12}}
\put(0,12){\line(1,0){12}}
\put(0,0){\line(0,1){12}}
\put(6,0){\line(0,1){12}}
\put(12,0){\line(0,1){12}}
\end{picture}\,$. Indeed, diagram (\ref{shiny_new_diagram}) and the formula
(\ref{Killing_curvature}) show this subbundle to be the range of the homomorphism
\begin{equation}\label{fancyR}
\mathcal R:E=\begin{array}{c}\Wedge^1\\[-3pt] \oplus\\[-1pt]
\Wedge^2\end{array}\ni\left[\!\begin{array}{c}\sigma_c\\ \mu_{cd}
\end{array}\!\right]\longmapsto
2R_{ab}{}^e{}_{[c}\mu_{d]e}+2R_{cd}{}^e{}_{[a}\mu_{b]}
-(\nabla^eR_{abcd})\sigma_e\in\begin{picture}(12,12)(0,2)
\put(0,0){\line(1,0){12}}
\put(0,6){\line(1,0){12}}
\put(0,12){\line(1,0){12}}
\put(0,0){\line(0,1){12}}
\put(6,0){\line(0,1){12}}
\put(12,0){\line(0,1){12}}
\end{picture}\,.\end{equation}
It follows that, when \( \kappa \) has constant rank, the composition
$$\begin{picture}(12,6)(0,0)
\put(0,0){\line(1,0){12}}
\put(0,6){\line(1,0){12}}
\put(0,0){\line(0,1){6}}
\put(6,0){\line(0,1){6}}
\put(12,0){\line(0,1){6}}
\end{picture}\xrightarrow{\,{\mathcal{C}}\,}
\begin{picture}(12,12)(0,2)
\put(0,0){\line(1,0){12}}
\put(0,6){\line(1,0){12}}
\put(0,12){\line(1,0){12}}
\put(0,0){\line(0,1){12}}
\put(6,0){\line(0,1){12}}
\put(12,0){\line(0,1){12}}
\end{picture}\xrightarrow{\,\phantom{\mathcal{C}}\,}
\overline{\begin{picture}(12,12)(0,2)
\put(0,0){\line(1,0){12}}
\put(0,6){\line(1,0){12}}
\put(0,12){\line(1,0){12}}
\put(0,0){\line(0,1){12}}
\put(6,0){\line(0,1){12}}
\put(12,0){\line(0,1){12}}
\end{picture}}
\equiv
\displaystyle
\begin{picture}(12,12)(0,2)
\put(0,0){\line(1,0){12}}
\put(0,6){\line(1,0){12}}
\put(0,12){\line(1,0){12}}
\put(0,0){\line(0,1){12}}
\put(6,0){\line(0,1){12}}
\put(12,0){\line(0,1){12}}
\end{picture}\,/{{\mathcal{R}}(E)}
\,,$$
is a well defined differential operator  of vector bundles.
We write ${\mathcal{L}}$ for the composition.
%
Based on these observations, in  \S\ref{kill_and_more_kill} we will establish the following.
\begin{theorem}\label{shiny_new_theorem}
  For any semi-Riemannian manifold such that
  the homomorphism $\mathcal R$ in~(\ref{fancyR}) has constant rank,
  the complex of linear differential operators
\begin{equation}\label{key_complex}\begin{picture}(6,6)(0,0)
\put(0,0){\line(1,0){6}}
\put(0,6){\line(1,0){6}}
\put(0,0){\line(0,1){6}}
\put(6,0){\line(0,1){6}}
\end{picture}\xrightarrow{\,{\mathcal{K}}\,}
\begin{picture}(12,6)(0,0)
\put(0,0){\line(1,0){12}}
\put(0,6){\line(1,0){12}}
\put(0,0){\line(0,1){6}}
\put(6,0){\line(0,1){6}}
\put(12,0){\line(0,1){6}}
\end{picture}\xrightarrow{\,{\mathcal{L}}\,}
\overline{\begin{picture}(12,12)(0,2)
\put(0,0){\line(1,0){12}}
\put(0,6){\line(1,0){12}}
\put(0,12){\line(1,0){12}}
\put(0,0){\line(0,1){12}}
\put(6,0){\line(0,1){12}}
\put(12,0){\line(0,1){12}}
\end{picture}}\,,\end{equation}
is (locally) exact if and only if the same is true of~
(\ref{curvature_reduced_complex}) for the Killing connection.
\end{theorem}
From equation~(\ref{Killing_curvature}) for the curvature $\kappa$ of the Killing connection it is evident that  that assumption on the rank of $\mathcal R$ is constant is equivalent to the rank of $\kappa$ being constant.

The above reasoning was the basis for our results in~\cite{CostanzaEastwoodLeistner21}, in which we
investigated the local exactness of the complex (\ref{key_complex}) for
{\em Rieman\-nian\/} locally symmetric spaces (meaning that $\nabla_aR_{bcde}=0$).  We showed that
(\ref{key_complex}) is always locally exact in the indecomposable case and,
for products
$$M=M_1\times M_2\times \cdots\times M_k$$
of indecomposables, we showed that (\ref{key_complex}) is locally exact unless
$M$ has at least one flat factor and at least one Hermitian factor, in which
case (\ref{key_complex}) fails to be locally exact.

For a general semi-Rieman\-nian locally symmetric space, indecomposable or not, similar
arguments, using our operator ${\mathcal{D}}$ and the augmented curvature,
reproduce the conclusions of Gasqui and Goldschmidt~\cite{GasquiGoldschmidt83}, and especially
Th\'eor\`eme 7.2, which identifies the local range of the Killing operator with
the kernel of a suitable {\em third order\/} linear operator
$$\begin{picture}(12,6)(0,0)
\put(0,0){\line(1,0){12}}
\put(0,6){\line(1,0){12}}
\put(0,0){\line(0,1){6}}
\put(6,0){\line(0,1){6}}
\put(12,0){\line(0,1){6}}
\end{picture}\longrightarrow
\overline{\begin{picture}(18,12)(0,2)
\put(0,0){\line(1,0){12}}
\put(0,6){\line(1,0){18}}
\put(0,12){\line(1,0){18}}
\put(0,0){\line(0,1){12}}
\put(6,0){\line(0,1){12}}
\put(12,0){\line(0,1){12}}
\put(18,6){\line(0,1){6}}
\end{picture}}\oplus
\overline{\begin{picture}(12,12)(0,2)
\put(0,0){\line(1,0){12}}
\put(0,6){\line(1,0){12}}
\put(0,12){\line(1,0){12}}
\put(0,0){\line(0,1){12}}
\put(6,0){\line(0,1){12}}
\put(12,0){\line(0,1){12}}
\end{picture}}\,.$$

In the {\em Lorentzian\/} case, an indecomposable symmetric
space is of constant sectional curvature or universally covered by a  Cahen--Wallach
space~\cite{Cahen-Wallach70}.  Since symmetric spaces are complete,  a simply connected symmetric space is a semi-Rieman\-nian  product of indecomposable ones \cite{derham52,wu64}.
Therefore, a Lorentzian locally symmetric space is locally isometric to the product of indecomposable Rieman\-nian symmetric spaces and a Lorentzian space that is either of constant sectional curvature or  a  Cahen--Wallach space. This is the {\em local de~Rham-Wu decomomposition} of a Lorentzian locally symmetric space.
In this article, we show that
(\ref{key_complex}) is locally exact on Cahen--Wallach spaces and, in general,
we prove the following.
\begin{theorem}\label{calabi-theo}
Suppose
$$M=M_1\times M_2\times \cdots\times M_k$$
is the local de~Rham-Wu decomposition of a Lorentzian locally symmetric space.  Then
(\ref{key_complex}) is locally exact unless there is at least one Hermitian
factor and at least one factor that is either flat or a Cahen--Wallach space, in
which case local exactness fails.
\end{theorem}
The proof of this theorem will be provided in Section~\ref{lorsec} by tying  together many results obtained throughout the article.
Although~\cite{CostanzaEastwoodLeistner21} deals with Rieman\-nian metrics and this article with
Lorentzian, much of our reasoning applies to general signature.  In particular,
Section~\ref{sec:products} is valid in arbitrary signature and there are just
two unresolved issues in extending Theorem~\ref{calabi-theo} in general.
Firstly, it is less clear how to obtain explicit replacements for the
Cahen--Wallach spaces.  Secondly, our current arguments break down for general
products, even for a product of two Cahen--Wallach spaces.

A strong motivation for identifying the range of the Killing operator (on Lorentzian manifolds)
comes from general relativity as follows.  For a vector field $X^a$, we observe
that
$$\textstyle({\mathcal{K}}X)_{ab}\equiv\nabla_{(a}X_{b)}
=\frac12{\mathcal{L}}_Xg_{ab},$$
where ${\mathcal{L}}_X$ is the Lie derivative.  Hence, the range of
${\mathcal{K}}$ may be regarded as the infinitesimal changes in the metric
$g_{ab}$ under the infinitesimal co\"ordinate changes, namely the first order
flows of vector fields.  As a physical theory, these co\"ordinate changes
should have no effect: they should be regarded as mere {\em gauge changes\/}
and, indeed, it is convenient to formulate {\em linearised gravity\/} by
discarding them from a metric of the form $\eta_{ab}+\epsilon h_{ab}$, where
$\eta_{ab}$ is the {\em flat\/} Minkowski metric and $h_{ab}$ is an arbitrary
symmetric tensor.  Instead, one could start with a {\em curved} background
metric and a locally symmetric metric is the next most reasonable choice.  So,
our task is to identify the {\em gauge freedom\/} in this setting (as further
expounded in \cite[5.7.11]{penroserindler1} and \cite[C.2.17]{wald84}).
Similar gravitational motivation lies behind
\cite{AksteinerAnderssonBackdahlKhavkineWhiting21,Khavkine19}, which indentify
the range of the Killing operator on a Schwarzschild or Kerr background.

As noted in~\cite{E}, the Killing operator
${\mathcal{K}}:\Wedge^1\to\bigodot^2\!\Wedge^1$ enjoys a hidden invariance,
which places it in the realm of projective differential geometry.
Specifically, there is an invariant differential operator
$\Wedge^1(2)\to\bigodot^2\!\Wedge^1(2)$, where $\Wedge^1(2)$ denotes the bundle
of $1$-forms with projective weight $2$, which may be identified with
${\mathcal{K}}$ in the presence of a metric.  It is the
operator~\cite[Equation~(2.3)]{E} and what we call here the {\em Killing
connection\/} appears as an invariant modification of the projectively
invariant {\em tractor connection\/} in~\cite[Section~4]{E}.  As shown by
Hammerl, Somberg, Sou\v{c}ek, and
\v{S}ilhan~\cite{HammerlSombergSoucekSilhan12}, this is part of a general
theory of {\em prolongation connections\/} for an extensive class of
overdetermined operators known as {\em first BGG operators\/} in the general
theory of {\em parabolic differential geometry\/}~\cite{cap-slovak-book1}.  By
this route, our general results on the range of a connection can be applied to
the class of first BGG operators.  In particular, the commutativity of the
bottom left-hand square of (\ref{shiny_new_diagram}) appears as a key feature
of the prolongation
connection~\cite[Corollary~3.1]{HammerlSombergSoucekSilhan12}.  Although the
Killing connection is not constructed by Hammerl-Somberg-Sou\v{c}ek-\v{S}ilhan
in their follow-up article~\cite{HammerlSombergSoucekSilhan12inv}, it would be
for the bundle $\Wedge^2({\mathcal{E}}_C)$ in their
notation~\cite[Section~3]{HammerlSombergSoucekSilhan12inv}.  A yet more general
theory of prolongation for overdetermined operators is constructed
in~\cite{BransonCapEastwoodGover06}.  As expected, a connection emerges.  It
lacks the invariance of~\cite{HammerlSombergSoucekSilhan12} but, nevertheless,
our results in this article on the range of a general connection are surely
applicable for the geometric systems considered
in~\cite{BransonCapEastwoodGover06}.

\newcommand{\Riem}{\mathcal{R}}
\newcommand{\Calabi}{\operatorname{D}}

Our article is organised as follows.
\tableofcontents

\section{The range of a connection on a vector bundle}\label{ConnectionSection}
When convenient, we employ the Penrose abstract indices notation as in \cite{penroserindler1}.
On a manifold, Latin indices \( a, b, c, \ldots \) stand for (abstract)  indices of $TM$ or $\Lambda^1$, whereas Greek indices \( \alpha, \beta, \gamma, \ldots \) stand for indices of a given vector bundle.
Under this convention, for example, an object of indices \( X_a{}^{\alpha}{}_{\beta} \) is a section of \( \Wedge^1 \otimes E \otimes E^* \).
Enclosing indices in round brackets \( (a \ldots b) \) means to take the symmetric component, and square brackets \( [a \ldots b] \) means to take the skew-symmetric component. For example, $\omega_{ab}\in \Wedge^2$ if and only if $\omega_{ab}=\omega_{[ab]}$ and $h_{ab}\in \bigodot^2\!\Wedge^1$ if and only if $h_{ab}=h_{(ab)}$.

\subsection{Connections on vector bundles}

For a connection  $\D:E\to \Wedge^1\otimes E$ on a vector bundle  $E$, we study the following problem, primarily restricting to the local setting:
\begin{question}\label{que: Range of connection?}
  When is a section \( \tensor{\phi}{_a^\alpha} \) of \( \Wedge^1\otimes E \) in the range of \( \connection \), i.e.~when is there a section \( \phi^\alpha \) of \( E \) such that
  $\connection_a \phi^\alpha =\tensor{\phi}{_a^\alpha}$?
\end{question}
Recall that a connection defines the {\em coupled de~Rham sequence}
\[ \begin{tikzcd}
  E \ar[r, "\connection"] & \Wedge^1 \otimes E \ar[r, "\dconn"] & \Wedge^2 \otimes E \ar[r, "\dconn"] & \Wedge^3 \otimes E \ar[r, "\dconn"] & \cdots
\end{tikzcd} , \]
where $\dconn$ denotes the  exterior covariant derivative,
which can be expressed concisely using index notation,
\[ \dconn_a \phi^{}_{b \ldots c}{}^\alpha := \connection^{}_{[a}\phi^{}_{b \ldots c]}{}^\alpha \quad \mbox{ for any \(E\)-valued \(p\)-form } \quad \phi_{b \ldots c}{}^{\alpha} \in \Gamma(\Wedge^p \otimes E) . \]
Here the derivative \( \connection_{a}\phi_{b \ldots c}{}^\alpha \) is defined by fixing an auxilliary choice of torsion-free connection \( \nabla_a \) on \( \Wedge^1 \).
The induced connection \( \nabla_a \) on \( \Wedge^p \) and then \( \connection_a \) on \( \Wedge^p \otimes E \) depend on this choice, but due to the skew-symmetrization, \( \dconn{}_a \phi{}_{b \ldots c}{}^\alpha \) does not.
See for example \cite[Section 1.3.1]{cap-slovak-book1} for more details.
Note that we have made  a choice of convention  here by including a factor of $1/p!$ 
in the skew symmetrisation,~e.g. $ \connection^{}_{[a}\phi^{}_{b]}=\tfrac12\left(
 \connection^{}_{a}\phi^{}_{b}-  \connection^{}_{b}\phi^{}_a\right)$,
which differs from \cite{cap-slovak-book1}, and instead follows~\cite{penroserindler1}.

In general, the coupled de~Rham sequence is not a complex.
Rather, the composition
\[\kappa := \dconn \D:E \xrightarrow{\qquad}   \Wedge^2 \otimes E \]
defines the curvature of \( \connection \), which is easily checked to be an algebraic operator.
One may consider \( \kappa \) as a section \( \tensor{\kappa}{_a_b^\alpha_\beta} \) of \( \Wedge^2 \otimes \End(E) \).
Writing
$\phi^\alpha$ for an arbitrary section of~$E$, the curvature is expressed in index notation as
\begin{equation}\label{eq-curvature of connection}
  \kappa_{ab}{}^\alpha{}_\beta\phi^\beta = \tfrac{1}{2}(\D_a\D_b-\D_b\D_a)\phi^\alpha .
\end{equation}
Note that the curvature of a connection on \( E \) with conventions as here differs from another common definition, the anti-commutator of derivatives, by a largely irrelevant factor of~\( \tfrac{1}{2} \).  For the Riemannian curvature tensor $R$, we follow the standard convention and define  it as $R=2\kappa$, where $\kappa=\nabla^\Wedge\circ \nabla$, for $\nabla$ the Levi-Civita connection on $TM$.

It is straightforward to check that the  compositions
\[ \kappa^{(k)} := (\dconn)^2\colon \Wedge^k \otimes  E \xrightarrow{\qquad}   \Wedge^{k+2} \otimes E \]
 are also homomorphisms, and are given by formula
\[ \kappa^{(k)}(\tensor{\phi}{_{c \ldots d}^\alpha}) = \tensor{\kappa}{_{[ab}^\alpha_{|\beta|}}\tensor{\phi}{_{c \ldots d]}^\beta},  \]
where $|\beta|$ indicates that the index $\beta$ is excluded from the skew-symmetrisation.
Consequently, the coupled de~Rham sequence of a connection is a differential complex if  the connection $\connection$ is {\em flat}.
Conversely, if the coupled de~Rham sequence is a differential complex, then the connection is flat.
Moreover, with a flat connection $D$, locally $E$ can be trivialised by  $D$-parallel sections, and hence the coupled de~Rham sequence reduces to $\mathrm{rank}(E)$ copies of  the usual de~Rham complex. It is, therefore,  locally exact.
Hence, for flat connections Question \ref{que: Range of connection?} has a well-known answer:
a section \( \phi_a{}^{\alpha} \) is locally in the range of \( \connection \) if and only if the section is closed, \( \dconn_a \phi^{}_b{}^{\alpha} = 0 \).

Even for connections with curvature, the coupled de~Rham sequence provides a necessary condition for a section \( \phi_a{}^{\alpha} \) to be in the range of \( \connection \).
If \( \phi_a{}^{\alpha} = \connection_a \phi^{\alpha} \) for some \( \phi^{\alpha} \in \Gamma(E) \), then a further application of \( \dconn \) gives
\begin{equation}\label{eq-curature sufficient condition}
  \dconn_a \phi^{}_b{}^{\alpha} = \dconn_a \connection^{}_b \phi^{\alpha} = \kappa_{ab}{}^{\alpha}{}_{\beta} \phi^{\beta} ,
\end{equation}
which is to say that \( \dconn_a \phi^{}_b{}^{\alpha} \) is in the range of the curvature operator.
This is a condition that is straightforward to verify in practice, by taking one derivative and checking linear algebraic conditions.
However, it is not always a sufficient condition.

To improve on the situation, observe that
\begin{equation}\label{eq: pre-Bianchi}
  \dconn\circ \kappa^{(k)} = (\dconn)^3 = \kappa^{(k+1)}\circ \dconn
\end{equation}
The \emph{Bianchi identity} \( \connection_{[a} \kappa_{bc]}{}^\alpha{}_\beta = 0 \) follows from skewing the indices \( abc \) in
\[ \connection_a \kappa{}_{bc}{}^\alpha{}_\beta \phi^\beta = (\connection_a \kappa{}_{bc}{}^\alpha{}_\beta) \phi^\beta + \kappa{}_{bc}{}^\alpha{}_\beta \connection_a \phi^\beta \quad \mbox{ for } \quad \phi^\alpha \in \Gamma(E)  \]
and comparing this to equation~(\ref{eq: pre-Bianchi}) for $k=0$.
From \eqref{eq: pre-Bianchi} also follows the existence of a well defined sub-sequence of the coupled de~Rham sequence, namely
\[ \begin{tikzcd}
0 \ar[r, "\connection"] & 0 \ar[r, "\dconn"] & \kappa(\sect{E})
\ar[r, "\dconn"] & \kappa^{(1)}(\sect{\Wedge^1 \otimes E}) \ar[r, "\dconn"] & \ldots
\end{tikzcd} . \]
\begin{remark*}
Here and from now on, whenever necessary, we suppose that the curvature \( \kappa \) and the induced homomorphisms $\kappa^{(k)}$ have constant rank, so that  \( \kappa^{(k)}(\Wedge^k \otimes E) \) is itself a vector bundle.
Since the rank is upper semi-continuous, this is an assumption that always holds on the connected components of an open dense set of $M$, and statements that require \( \kappa^{(k)}(\Wedge^k \otimes E) \) to be a vector bundle can be viewed as  being valid on such a connected component.
This restriction implicitly applies to   statements in this  and the following section that involve the kernel and the image of the curvature, or related quotients. In Section~\ref{kill_and_more_kill}, where we prove Theorem~\ref{shiny_new_theorem}, it is explicitly assumed. In Sections~\ref{sec:products} and~\ref{lorsec} the assumption is satisfied on all of $M$ because of local symmetry and therefore allows us to use the all earlier statements without restriction.
\end{remark*}

The quotient sequence
\begin{equation}\label{eq: gauge complex}
  \begin{tikzcd}
    \sect{E} \ar[r, "\connection"] & \sect{\Wedge^1 \otimes E} \ar[r, "\dconn"] & \sect{\frac{\Wedge^2 \otimes E}{\kappa(E)}}
    \ar[r, "\dconn"] & \sect{\frac{\Wedge^3 \otimes E}{\kappa^{(1)}(\Wedge^1 \otimes E)}} \ar[r, "\dconn"] & \ldots
  \end{tikzcd}
\end{equation}
is, by construction, a complex.
Denote the \( k \)-th homology of this complex by \( H^k(E, \connection) \).

This homology provides an answer to question \ref{que: Range of connection?}, generalising the flat case.
Indeed, it is easily seen that a section \( \phi \) of \( \Wedge^k \otimes E \) is in the range of \( \dconn \) if and only if it represents the zero homology class in \( H^k(E, \connection) \).
\cblock{If \( \eta \) is in the range of connection, then applying connection again get element of range of curvature, so a class in homology, which is then manifestly exact. Conversely, if \( \eta \) represents a zero class, then up to range of curvature is in range of connection, but curvature is also range of connection.}

In particular, when the complex is exact at \( \Wedge^1 \otimes E \), the range of \( \connection \) comprises exactly those sections \( \phi \) of \( \Wedge^1 \otimes E \) for which \( \dconn \phi = \kappa(\psi) \) for some \( \psi \in \Gamma(E) \).
We observe but do not pursue here the fact that homology further along the complex can be used to answer similar questions.
This suggests the following definition.
\begin{definition}\label{exact_connection}
 A connection \( \connection \) is \emph{exact at \( \Wedge^1 \otimes E \)},
 or \emph{exact} for short, if for all
 \( \phi \in \Gamma(\Wedge^1 \otimes E) \) such that
 \( \dconn \phi \) is in the range of \( \kappa \), one has
 \( \phi \) in the range of \( \connection \).
\end{definition}
We should point out that for an exact connection, as just defined, the coupled de Rham sequence is not necessarily a locally exact complex, which, as explained above, is only the case for flat connections. Hence, in regards to the coupled de Rham sequence, an exact connection is only {\em exact up to curvature} or {\em exact modulo curvature}, but for brevity and since there is no danger of confusion with the flat case, we use the term {\em exact connection}.
Not all connections are exact, as can be seen for any non-flat connection on an arbitrary bundle over a surface.
On the other hand, flat connections are clearly exact.

Definition~\ref{exact_connection} says that the
complex of vector bundles (\ref{curvature_reduced_complex}) is exact.
Equivalently, if we set $E_0\equiv\ker(\kappa)$, then the complex
\begin{equation}\label{exactE0}
\begin{tikzcd}
  \sect{E_0} \ar[r, "\connection"] & \sect{\Wedge^1 \otimes E} \ar[r, "\dconn"] & \sect{\Wedge^2 \otimes E}
\end{tikzcd} \end{equation}
is exact. This is a consequence of the following proposition.
\begin{proposition}\label{thm: prop isomorphism between H1 and ker D}
  There is a canonical isomorphism
  \[ H^1(E,\connection) \cong \ker(\dconn) / \operatorname{im}(\connection|_{E_0}) . \]
\end{proposition}
\begin{proof}
  For \( [\phi] \in H^1(E,\connection) \) and any representative \( \phi \) of \( [\phi] \), there is, by definition, some \( \psi \in \Gamma(E) \) so that \( \dconn \phi = \kappa(\psi) \).
  Although not uniquely specified, the difference between choices \( \psi, \psi' \) will be a section of \( E_0 \).
  Define the map \( f\colon [\phi] \mapsto \phi - \connection \psi \) into \( \ker(\dconn) / \operatorname{im}(\connection|_{E_0}) \),
  which is clearly independent of choices made.

  That  \( f \) has an inverse  follows by noting that an element \( \phi \) in the kernel of \( \dconn \) is also contained in the kernel
  \[ \ker\bigl( \dconn \colon \Wedge^1 \otimes E \to \Wedge^2 \otimes E / \kappa(E) \bigr) , \]
  so represents an element of \( H^1(E,\connection) \).
\end{proof}

The content of the proof is that any element of \( H^1(E,\connection) \) may be represented by an element of the kernel of \( \dconn \), and uniquely so if \( E_0 = 0 \), i.e. \( \kappa \) is injective.
This is a useful maneuver, which will be employed often below.
\begin{proposition}\label{Dinj-prop}
The connection $D$ is exact if and only if the complex (\ref{exactE0}) is exact. In particular,
if $\dconn:\Wedge^1\otimes E \to \Wedge^2\otimes E$ is injective, then $\D$ is exact. Conversely, if $\D$ is exact and the curvature $\kappa: E\to \Wedge^2\otimes E$ is injective, then $\dconn:\Wedge^1\otimes E \to \Wedge^2\otimes E$ is injective.
\end{proposition}

We will say that a subbundle \( F \subseteq E \) is \emph{parallel} for \( \connection \) if
\[ \connection( \Gamma(F) ) \subseteq \Gamma(\Wedge^1 \otimes F) . \]
On a parallel subbundle $F$, the connection $D$ induces a connection by restriction, which we denote by $D|_F$.

\begin{corollary}\label{coro:ker_D_exact}
  Let $E$ be a vector bundle with connection $D$ and suppose that $F \subseteq E$ is a parallel sub-bundle such that $D|_{F}$ is exact. If $\ker(D^{\wedge}) \subseteq \Gamma(\Wedge^1\otimes F)
  $, then $D$ is exact.
\end{corollary}
		\begin{proof}
			Let $F_{0} = F \cap E_{0}$ denote the kernel of $\kappa$, restricted to $F$. Under the assumption that $D|_{F}$ is exact, we know that $\ker((D|_{F})^{\wedge}) = \mathrm{Im}(D|_{F_{0}})$, by Proposition \ref{Dinj-prop}. Then, if $\ker(D^{\wedge}) \subseteq \Gamma(\Wedge^1\otimes F)$,  it follows that
			\[
			\ker(D^{\wedge}) \subseteq \ker((D|_{F})^{\wedge}) = \mathrm{Im}(D|_{F_{0}}) \subseteq \mathrm{Im}(D|_{E_{0}}) \subseteq \ker(D^{\wedge}).
			\]
			This means that $\ker(D^{\wedge}) = \mathrm{Im}(D|_{E_{0}})$ and consequently $H^{1}(E, D)
			= \lbrace 0 \rbrace$ by Proposition~\ref{thm: prop isomorphism between H1 and ker D}.
		\end{proof}

One more observation is that exactness behaves well under direct sums.
\begin{proposition}\label{product-prop}
Given bundles \( E_1, E_2 \) over $M$ with connections \( \connection_1, \connection_2 \),
form the direct sum \( (E,\connection) = (E_1 \oplus E_2, \connection_1 + \connection_2) \).
The connection $\connection$ is exact if and only if  $\connection_1$ and \( \connection_2 \) are both exact.
\end{proposition}
\begin{proof}
Clearly \( \kappa^{(k)}(\Wedge^k \otimes E) = \kappa^{(k)}_1(\Wedge^k \otimes E_1) \oplus \kappa^{(k)}_2(\Wedge^k \otimes E_2) \), so the quotient complex of Equation \eqref{eq: gauge complex} splits, and thus \( H^k(E, \connection) = H^k(E_1, \connection_1) \oplus H^k(E_2, \connection_2) \).
\end{proof}

We note that a connection \( (E,\connection) \) is equivalent to \( (E_1 \oplus E_2, \connection_1 + \connection_2) \) if and only if the summands \( E_1, E_2 \) of \( E \) are parallel, in which case \( D_i = D|_{E_i} \) for \( i = 1,2 \).

\begin{proposition}\label{prop:pullback_exactness}
	Let $D$ and $\bar{D}$ be exact connections on the vector bundles $E\to M$ and $\bar{E}\to \bar M$ over $M$ and $\bar{M}$  and denote by $\pi$ and $\bar \pi$ the natural projections from the product manifold $M\times \bar M$ to $M$ and $\bar M$ respectively.
	 Then, $\pi^{*}D$ and $\bar\pi^*\bar D$ are exact connections on the respective vector bundles $\pi^{*}E$ and $\bar\pi^*\bar E$ over $M\times \bar M$. Moreover,  $\pi^{*}D + \bar\pi^{*}\bar{D}$ is an exact connection on the vector bundle $\pi^{*}E\oplus \bar\pi^{*}\bar E$ over $M\times \bar M$.
\end{proposition}
\begin{proof}
First we consider $\pi^*D $ on $\pi^*E$ over $M\times \bar M$.
	With the usual conventions, we split the one-forms on $M\times \bar M$ as the sum of the one-forms on the factors, $\Wedge^1=
	\Wedge^{1,0}\+	\Wedge^{0,1}$ and similarly for the two-forms, $	\Wedge^2=	\Wedge^{2,0}\+	\Wedge^{1,1}\+	\Wedge^{0,2}$.
	We also use unbarred and barred Latin indices for the $\Wedge^{1,0}$ and  $\Wedge^{0,1}$ components, and as before Greek indices for the $E$ components.
	Then
	we are required to consider the
diagram
\begin{center}\begin{picture}(160,60)
\put(0,30){\makebox(0,0){$\pi^*E$}}
\put(70,40){\makebox(0,0){$\Wedge^{1,0}\otimes\pi^*E$}}
\put(70,20){\makebox(0,0){$\Wedge^{0,1}\otimes\pi^*E$}}
\put(160,50){\makebox(0,0){$\Wedge^{2,0}\otimes\pi^*E$}}
\put(160,30){\makebox(0,0){$\Wedge^{1,1}\otimes\pi^*E$}}
\put(160,10){\makebox(0,0){$\Wedge^{0,2}\otimes\pi^*E$.}}
\put(17,32){\vector(3,1){20}}
\put(17,28){\vector(3,-1){20}}
\put(105,42){\vector(3,1){20}}
\put(105,38){\vector(3,-1){20}}
\put(105,22){\vector(3,1){20}}
\put(105,18){\vector(3,-1){20}}
\end{picture}\end{center}
The composition along the top is
$\pi^*\kappa:\pi^*E\to\Wedge^{2,0}\otimes\pi^*E$ and exactness of $D_a$ implies
that
$$\pi^*E\longrightarrow\Wedge^{1,0}\otimes\pi^*E\longrightarrow
\frac{\Wedge^{2,0}\otimes\pi^*E}{\pi^*\kappa(\pi^*E)}$$
is locally exact.  Without loss of generality, we may therefore assume that we
are given just $\psi_{\bar a}{}^\alpha\in\Gamma(\Wedge^{0,1}\otimes\pi^*E)$ in the kernel of $\pi^*D^\wedge$, i.e.~such that
$$D_a\psi_{\bar a}{}^\alpha=0\quad\mbox{and}\quad
\partial_{[\bar a}\psi_{\bar b]}{}^\alpha=0,$$
and we want locally to find $\phi^\alpha\in\Gamma(\pi^*E)$ such that
$$D_a\phi^\alpha=0\quad\mbox{and}\quad
\partial_{\bar a}\phi^\alpha=\psi_{\bar a}{}^\alpha.$$
Our second assumption $\partial_{[\bar a}\psi_{\bar b]}{}^\alpha=0$ says that
$\psi_{\bar a}{}^\alpha$ is closed as a $1$-form along the fibres of~$\pi$.  At
any particular point $p\in M$ this form has values in $E_p$, the fibre of $E$
over~$p$, but this plays no r\^ole in concluding that locally we may write
$\psi_{\bar a}{}^\alpha=\partial_{\bar a}\phi^\alpha$.  Indeed, following
standard procedure, we may construct $\phi^\alpha$ as the integral along a
curve $\gamma\hookrightarrow\pi^{-1}(p)$, the only ambiguity being the choice
of basepoint from which to start this curve.  Moreover, since we are working on
a product manifold $M\times\bar{M}$, we can use the same
$\gamma\hookrightarrow\bar{M}$ as $p\in M$ varies.  Therefore we may
`differentiate under the integral sign' to conclude that $D_a\phi^\alpha=0$, as
required.
The same proof applies to $\bar\pi^*\bar D$ and the final result follows from Proposition~\ref{product-prop} applied to $E_1=\pi^*E$ and $E_2=\bar\pi^*\bar E$, both over $M\times \bar M$.
\end{proof}

Generally, given an arbitrary subbundle \( F \) that is parallel for \( \connection \), the connection both restricts to a connection on \( F \) and descends to a well defined connection on the quotient \( E / F \).
In the following we denote both the restriction and quotient connections by \( \connection \), and use \( \kappa \) for their respective curvatures.
\begin{proposition}\label{thm: sub-quotient exactness}
  Given a connection \( (E, \connection) \) and a parallel sub-bundle \( F \), if \( F \) and \( E / F \) are exact, and if the curvature on \( E/F \) is injective, then \( E \) is exact.
\end{proposition}
\begin{proof}

  Given \( \eta \in H^1(E,\connection) \), fix a representative section \( \eta \) of \( \Wedge^1 \otimes E \) such that \( \dconn(\eta) = 0 \).
  Descending to the quotient \( E/F \), we have that \( \dconn([\eta]) = [\dconn(\eta)] = 0 \).
  By the hypotheses on \( E/F \) and Proposition \ref{thm: prop isomorphism between H1 and ker D} it follows \( \dconn \) is injective on \( E/F \), so that \( [\eta] = 0 \).
  In other words, \( \eta \in \Gamma(\Wedge^1 \otimes F) \).
  From exactness of \( F \) it follows that \( \eta \) is in the range of \( \connection|_F \), which suffices.
\end{proof}

\subsection{Genericity}
\label{gensec}
Here is a sufficient condition for a connection to be exact.
\begin{definition}
  A connection $\connection$ is
 \emph{generic} if \( \kappa^{(1)} \colon \Wedge^1 \otimes E \to \Wedge^3 \otimes E\) is injective.
\end{definition}
For connections on any manifold of dimension more than 3, one expects that a connection will be generic by dimensional considerations.
In dimensions 2 and 3, no connection satisfies this genericity condition.

\begin{lemma}\label{genlemma}\label{prop-generic}
  Suppose that a connection \( (E, \connection) \) is generic. Then
  \begin{enumerate}
    \item
    \( \connection \) has injective  curvature, and

    \item
    \( \connection \) is exact.
  \end{enumerate}
\end{lemma}
\begin{proof}
  (1)
  For any decomposable \( \omega \otimes e \in \Gamma(\Wedge^1 \otimes E) \), the map \( \kappa^{(1)} \) is given by skewed product of \( \omega \) and the \( E \)-valued \( 2 \)-form \( \kappa(e) \).
  If $\kappa$ were not injective, then for any nonzero $e$ in its kernel and any 1-form \( \omega \), one would have \( \omega \otimes e \) in the kernel of \( \kappa^{(1)} \).

  (2)
  Due to the obvious
  \[ \ker\left( \dconn \right) \subseteq \ker(\kappa^{(1)}), \]
  exactness follows from Proposition \ref{thm: prop isomorphism between H1 and ker D}.
\end{proof}

Genericity provides a simple test for exactness of a connection.
On the other hand, there are examples of interest that are exact but not generic, such as any exact connection on a 3-dimensional manifold, or the connection defined on Cahen--Wallach spaces below.
The following sections develop more general theory, able to treat these examples.

\subsection{Augmented curvature}\label{subsection_augmented}

Let us return to the de~Rham sequence,
and the necessary condition for
$\phi\in\Gamma(\Wedge^1\otimes E)$ to be in the range of
$D:E\to\Wedge^1\otimes E$, namely that $\dconn\phi$ be in the range of
$\kappa \colon E\to\Wedge^2\otimes E$. This condition is useful
because the curvature $\kappa$ is a homomorphism of vector bundles,
and its range can usually be identified explicitly as a subbundle of
$\Wedge^2\otimes E$.
However, we have seen that the condition is not always sufficient.

There is, however, an equally canonical tensorial bundle $\Delta^2$, equipped with
\begin{itemize}
\item a canonical surjection $\Delta^2\to\Wedge^2$, and
\item a canonically defined linear differential operator
${\mathcal{D}}:\Wedge^1\otimes E\to\Delta^2\otimes E$, so that
\item the composition $\Wedge^1\otimes E\xrightarrow{\,{\mathcal{D}}\,}
\Delta^2\otimes E\to\Wedge^2\otimes E$ is just~$\dconn$, and
\item the composition $E\xrightarrow{\,\D\,}\Wedge^1\otimes E
\xrightarrow{\,{\mathcal{D}}\,}\Delta^2\otimes E$ is a homomorphism of
vector bundles.
\end{itemize}
Since  \( \Delta^2 \) is larger than \( \Wedge^2 \), this construction provides a possibly stronger necessary condition
for $\phi\in\Gamma(\Wedge^1\otimes E)$ to be in the range of~$D$, and therefore
has the potential to provide a necessary and sufficient condition.
The only price to pay is that the operator ${\mathcal{D}}:\Wedge^1\otimes E\to\Delta^2\otimes E$ is second order.

The bundle $\Delta^2$ is most cleanly defined via jets. Recall the first jet
exact sequence for~$\Wedge^2$,
\begin{equation}\label{J1Wedge2}
0 \to \Wedge^1\otimes\Wedge^2\to J^1\Wedge^2\to\Wedge^2\to 0,\end{equation}
and note that bundle $\Wedge^1\otimes\Wedge^2$ canonically splits into \( \GL \)-irreducibles,
$$\Wedge^1\otimes\Wedge^2=\Wedge^3\oplus\;
\begin{picture}(12,12)(0,1)
\put(0,0){\line(1,0){6}}
\put(0,6){\line(1,0){12}}
\put(0,12){\line(1,0){12}}
\put(0,0){\line(0,1){12}}
\put(6,0){\line(0,1){12}}
\put(12,6){\line(0,1){6}}
\end{picture}\, ,$$
where $\begin{picture}(12,12)(0,1)
\put(0,0){\line(1,0){6}}
\put(0,6){\line(1,0){12}}
\put(0,12){\line(1,0){12}}
\put(0,0){\line(0,1){12}}
\put(6,0){\line(0,1){12}}
\put(12,6){\line(0,1){6}}
\end{picture}$ denotes the tensors $\psi_{abc}$ in $\Wedge^1\otimes\Wedge^2$ such that $\psi_{[abc]}=0$ (for the use of Young tableaux, see \cite{fulton-harris}).
Therefore, we may form the quotient bundle
$\Delta^2\equiv J^1\Wedge^2/\Wedge^3$, which fits into a short exact sequence
$$0\to\begin{picture}(12,12)(0,1)
\put(0,0){\line(1,0){6}}
\put(0,6){\line(1,0){12}}
\put(0,12){\line(1,0){12}}
\put(0,0){\line(0,1){12}}
\put(6,0){\line(0,1){12}}
\put(12,6){\line(0,1){6}}
\end{picture}\to\Delta^2\to\Wedge^2\to 0.$$
The exterior derivative $d:\Wedge^1\to\Wedge^2$ gives rise to a canonical `not-yet-coupled'
second order linear differential operator ${\mathcal{D}}:\Wedge^1\to\Delta^2$,
defined as the composition
$$\Wedge^1\xrightarrow{\,j^1\d}J^1\Wedge^2\to\Delta^2.$$
Any torsion-free connection $\nabla_a$ on $\Wedge^1$ induces a connection on
$\Wedge^2$, equivalently a splitting of~(\ref{J1Wedge2}), and such a choice
enables us to write down an explicit formula for ${\mathcal{D}}$, namely
$$\Wedge^1\ni\omega_c\longmapsto
\left[\!\begin{array}{c}\nabla_{[b}\omega_{c]}\\
\nabla_a\nabla_{[b}\omega_{c]}
\end{array}\!\right]\in\Delta^2.$$
That $\nabla_a$ is torsion-free implies $\nabla_{[a}\nabla_b\omega_{c]}=0$, so
$\nabla_a\nabla_{[b}\omega_{c]}$ already lies in
$\begin{picture}(12,12)(0,1)
\put(0,0){\line(1,0){6}}
\put(0,6){\line(1,0){12}}
\put(0,12){\line(1,0){12}}
\put(0,0){\line(0,1){12}}
\put(6,0){\line(0,1){12}}
\put(12,6){\line(0,1){6}}
\end{picture}$, and nothing is lost by passing to the quotient
$J^1\Wedge^2\to\Delta^2$. In any case, we obtain a locally exact sequence
$$0\to{\mathbb{R}}\to\Wedge^0\xrightarrow{\,\d\,}\Wedge^1
\xrightarrow{\,{\mathcal{D}}\,}\Delta^2$$
and it remains to write down a coupled
version~${\mathcal{D}}:\Wedge^1\otimes E\to\Delta^2\otimes E$.
We adopt the following formula for the coupled version of~${\mathcal{D}}$, with \( \kappa \) as in Equation \eqref{eq-curvature of connection},
$$\Wedge^1\otimes E\ni\phi_c{}^\alpha\stackrel{{\mathcal{D}}}{\longmapsto}
\left[\!\begin{array}{c}\D_{[b}\phi_{c]}{}^\alpha\\
\D_a\D_{[b}\phi_{c]}{}^\alpha-\kappa_{bc}{}^\alpha{}_\beta\phi_a{}^\beta
\end{array}\!\right]\in\Delta^2\otimes E,$$
where the second line vanishes on totally skewing \( abc \), so lies in
$\,\begin{picture}(12,12)(0,1)
\put(0,0){\line(1,0){6}}
\put(0,6){\line(1,0){12}}
\put(0,12){\line(1,0){12}}
\put(0,0){\line(0,1){12}}
\put(6,0){\line(0,1){12}}
\put(12,6){\line(0,1){6}}
\end{picture}\,$ as required.
It is readily verified that this operator ${\mathcal{D}}$
does not depend on choice of torsion-free affine connection.
\begin{proposition}\label{thm: prop augmented curvature formula}The composition
$$E\xrightarrow{\,\D\,}\Wedge^1\otimes E\xrightarrow{\,{\mathcal{D}}\,}
\Delta^2\otimes E$$
is a homomorphism of vector bundles, given explicitly by
$$\phi^\alpha\longmapsto
\left[\!\begin{array}{c}\kappa_{bc}{}^\alpha{}_\beta\phi^\beta\\
(\D_a\kappa_{bc}{}^\alpha{}_\beta)\phi^\beta
\end{array}\!\right].$$
\end{proposition}
\begin{proof}A straightforward computation.\end{proof}
\noindent Notice that we have inadvertently proved the Bianchi identity a second time:
\begin{corollary}We always have $(\D_a\kappa_{bc}{}^\alpha{}_\beta)\phi^\beta\in
\Gamma\big(\,\begin{picture}(12,12)(0,1)
\put(0,0){\line(1,0){6}}
\put(0,6){\line(1,0){12}}
\put(0,12){\line(1,0){12}}
\put(0,0){\line(0,1){12}}
\put(6,0){\line(0,1){12}}
\put(12,6){\line(0,1){6}}
\end{picture}\otimes E\big)$, equivalently $\D_{[a}\kappa_{bc]}{}^\alpha{}_\beta=0$.
\end{corollary}
\begin{definition}Call the homomorphism $E\to\Delta^2\otimes E$ the \emph{augmented curvature\/} of the connection~$\D$.
\end{definition}
\noindent We have shown that the augmented curvature is invariantly defined, and it is clear that the
composition
\[ \begin{tikzcd}
E \xrightarrow{\quad} \Delta^2\otimes E \xrightarrow{\quad} \Wedge^2\otimes E
\end{tikzcd} \]
is the ordinary curvature \( \kappa \). It follows that
$$\ker(\mbox{{augmented curvature}})\subseteq\ker(\mbox{ordinary curvature}) .$$
The immediate question of equality here is answered by the first case of
Lemma~\ref{dkappa-lemma} below.

Any section \( \phi_a{}^{\alpha} \) of \( \Wedge^1 \otimes E \) of the form \( \connection_a \phi^\alpha \) for some \( \phi^\alpha \) necessarily satisfies the condition \( \mathcal{D}(\phi_a{}^{\alpha})
= \mathcal{D}\connection(\phi^{\alpha}) \), which in expanded form is
\begin{equation}\label{eq-augmented curvature necessary condition}
  \left[\!\begin{array}{c}\D_{[b}\phi_{c]}{}^\alpha \\
  \D_a\D_{[b}\phi_{c]}{}^\alpha-\kappa_{bc}{}^\alpha{}_\beta\phi_a{}^\beta
  \end{array}\!\right]
  = \left[\!\begin{array}{c}\kappa_{bc}{}^\alpha{}_\beta\phi^\beta \\
  (\D_a\kappa_{bc}{}^\alpha{}_\beta)\phi^\beta
  \end{array}\!\right] .
\end{equation}
As with the condition provided by curvature (Equation \eqref{eq-curature sufficient condition}), this provides an algebraic criterion for being in the range of \( \connection \), and one potentially stronger.
Indeed, in case that the kernel \( E_0 \) of curvature is parallel, the augmented curvature provides a complete characterisation of the range of \( \connection \).
\begin{proposition}\label{complete_characterization}
  Suppose given a connection \( \connection \) with parallel curvature kernel \( E_0 \).
  Any \( \phi_a{}^{\alpha} \in \Gamma(\Wedge^1 \otimes E) \) is in the range of \( \connection \) if and only if Equation (\ref{eq-augmented curvature necessary condition}) holds for some \( \phi^\alpha \in \Gamma(E) \).
\end{proposition}
\begin{proof}
  Supposing that Equation \eqref{eq-augmented curvature necessary condition} holds, we have in particular that
  \[ \D_{[b}\phi_{c]}{}^\alpha = \kappa_{bc}{}^\alpha{}_\beta\phi^\beta . \]
  Feeding this back into \eqref{eq-augmented curvature necessary condition} leads to the conclusion that
  \[ \left[\!\begin{array}{c}\kappa_{bc}{}^\alpha{}_\beta\phi^\beta\\
  (\D_a\kappa_{bc}{}^\alpha{}_\beta)\phi^\beta
  \end{array}\!\right]
  +
  \left[\!\begin{array}{c}0\\
  \kappa_{bc}{}^\alpha{}_\beta(\D_a\phi^\beta-\phi_a{}^\beta)
  \end{array}\!\right] = \left[\!\begin{array}{c}\kappa_{bc}{}^\alpha{}_\beta\phi^\beta \\
  (\D_a\kappa_{bc}{}^\alpha{}_\beta)\phi^\beta
  \end{array}\!\right] \]
  and thus that
  \[ \kappa_{bc}{}^\alpha{}_\beta(\D_a\phi^\beta-\phi_a{}^\beta) = 0 . \]

  Let \( \psi_a{}^{\alpha} = \phi_a{}^\beta - \D_a\phi^\beta \), which on the one hand has just been seen to be an element of \( \Wedge^1 \otimes E_0 \),
  and on the other hand is in the range of \( \connection \) if and only if \( \phi_a{}^{\alpha} \) is.
  But \( E_0 \) is assumed parallel, so \( \connection \) restricts to a flat connection on \( E_0 \).
  It is easily checked that \( \psi_a{}^{\alpha} \) satisfies \( \dconn \psi_a{}^{\alpha} = 0 \),  so we may conclude that \( \psi_a{}^{\alpha} = \connection_a \psi^{\alpha} \) for some \( \psi^{\alpha} \).
\end{proof}

\subsubsection{An example: the round sphere}
A good place to see augmented curvature in action is on semi-Riemannian manifolds with constant sectional curvature. For simplicity, we consider  the unit $n$-sphere.
Let us take $E$ to be its tangent bundle and recall that the curvature tensor
for the Levi-Civita connection $\nabla_a$ is given by half the Rieman\-nian curvature,
$$\kappa_{ab}{}^c{}_d=\tfrac{1}{2}\delta_a{}^cg_{bd}-\tfrac{1}{2}\delta_b{}^cg_{ad},$$
where $g_{ab}$ is the metric. It follows that
$$\phi_c{}^d\stackrel{{\mathcal{D}}}{\longmapsto}
\left[\!\begin{array}{c}\nabla_{[b}\phi_{c]}{}^d\\
\nabla_a\nabla_{[b}\phi_{c]}{}^d+\phi_{a[b}\delta_{c]}{}^d
\end{array}\!\right],$$
whilst the augmented curvature is given by
$$\phi^d\longmapsto
\left[\!\begin{array}{c}\delta_{[b}{}^d\phi_{c]}\\
0\end{array}\!\right].$$
As such, in order locally to write a tensor
$\phi_c{}^d\in\Gamma(\Wedge^1\otimes TS^n)$ as $\nabla_c\phi^d$ for some
vector field~$\phi^d$, it is firstly necessary that
\begin{equation}\label{nec}
\nabla_{[b}\phi_{c]}{}^d=\delta_{[b}{}^d\phi_{c]}\end{equation} for some
$1$-form $\phi_c$ and, if this is the case, then a further necessary but also
sufficient condition is that
\begin{equation}\label{nec_and_suff}
\nabla_a\nabla_{[b}\phi_{c]}{}^d+\phi_{a[b}\delta_{c]}{}^d=0.\end{equation}
Notice that in two dimensions (\ref{nec}) is vacuous, but (\ref{nec_and_suff})
is non-trivial in any dimension, as can be seen by taking
$\phi_c{}^d=\delta_c{}^d$.

\section{Non-injective curvature and the curvature filtration}\label{sec-curvature filtration}
To deal with connections for which the curvature of $\D$ is not injective,
we build an ascending filtration of \( E \) from the generalised kernels of \( \kappa \): let
\[ E_0 = \ker\bigl( \kappa \colon E \xrightarrow{\quad} \Wedge^2 \otimes E \bigr) \]
 and inductively
\[ E_{r+1} = \left\{ \eta	 \in E \colon \kappa (\eta) \in \Wedge^2 \otimes E_r \right\} . \]
We call this the \emph{curvature filtration}.
Since $E$ is of finite rank, the sequence stabilises, say at \( r = R \) and we have  $E_R=E_{R+1}=\ldots$.
For notational consistency, take \( E_{-1} = 0 \).
We will continue to assume  that \( \kappa \) is of sufficiently constant rank that each \( E_i \) is a subbundle of \( E \).

We will be interested in connections whose curvature filtrations are parallel.
In such case, the quotient \( E/E_R \) with induced connection has injective curvature, or else \( E_R \) would not be maximal.
Note that there do exist connections for which the curvature filtration is not parallel.
For example, on the trivial rank 2 bundle over \( \R^2 \), fix co\"ordinates \( x_1, x_2 \), as well as independent sections \( e_1, e_2 \), and let \( \connection \) be the connection defined by
\begin{align*}
  \connection e_1 & = \d x_1 \otimes e_2 \\
  \connection e_2 & = x_2\d x_1 \otimes e_1 .
\end{align*}
It is straightforward to check that the kernel of curvature here is the subbundle spanned by \( e_1 \) and that \( \connection \) does not preserve this subbundle.

Another such example is a semi-Rieman\-nian manifold $(M,g)$ that admits a
homothetic vector field, i.e.~a vector field $X$ such that $\nabla
X=\mathrm{Id}$, where $\nabla$ is the Levi-Civita connection.  An example of
such a situation is given by the Euler vector field of the Rieman\-nian cone over
a Rieman\-nian manifold, where we assume that the cone is not flat.  Clearly, the
Rieman\-nian curvature tensor satisfies $X\hook R=0$ and $E_0=\R\cdot X$, but
$E_0$ is not parallel.

Before providing a characterization of connections with parallel curvature filtration (Lemma \ref{dkappa-lemma}), we turn to the consequences of having such.
So, assume given a connection \( \connection \) whose curvature filtration is parallel.
Then $\connection$ descends to each sub-quotient \( E_{r+1}/E_r \).
By construction, the following diagram commutes for each \( r \ge 0 \),
\[ \begin{tikzcd}
  E_{r} \ar[r, "\kappa"] \ar[d] & \Wedge^2 \otimes E_{r} \ar[d] \\
  E_{r} / E_{r-1} \ar[r, "\kappa"] & \Wedge^2 \otimes E_{r} / E_{r-1}
\end{tikzcd} \]
so the curvature on \( E_{r} / E_{r-1} \) is equal to $\kappa$ modulo \( \Wedge^2 \otimes E_{r-1} \).
But, by definition, the restriction of \( \kappa \) to $E_{r}$ takes values in \( \Wedge^2 \otimes E_{r-1} \), so in fact the induced connection on each bundle \( E_{r}/E_{r-1} \) is flat.
Furthermore, for \( r \ge 1 \), the map
\[ \begin{tikzcd}
\kappa \colon E_{r} / E_{r-1} \ar[r] & \Wedge^2 \otimes E_{r-1} / E_{r-2}
\end{tikzcd} \]
is easily checked to be well defined and injective.
More is true.
\begin{lemma}\label{BensLemma1}
If the curvature filtration of \( \connection \) is parallel, then
  for each \( r \ge 1 \), the map
  \begin{equation}\label{eq: ker of connection}
    \ker\left( \Wedge^1 \otimes E_{r} / E_{r-1} \xrightarrow{\; \dconn \;} \Wedge^2 \otimes E_{r} / E_{r-1} \right)
    \xrightarrow{\; \dconn \;} \Wedge^2 \otimes E_{r-1} / E_{r-2}
  \end{equation}
  is injective.
\end{lemma}
\begin{proof}
  The de~Rham complex of \( E_{r} / E_{r-1} \) is exact, so
  for \( \phi \) in the domain of \eqref{eq: ker of connection}, there exists \( \psi \in \Gamma(E_{r} / E_{r-1}) \) for which \( \connection \psi = \phi \).
  But if \( \dconn \phi = 0 \in \Gamma(\Wedge^2 \otimes E_{r-1} / E_{r-2}) \), then
  \[ \kappa(\psi) = \dconn \connection \psi = \dconn \phi = 0 \in \Gamma(\Wedge^2 \otimes E_{r-1} / E_{r-2}) , \]
  which is to say that \( \kappa(\psi) \in \Gamma(\Wedge^2 \otimes E_{r-2}) \), and so \( \psi \in \Gamma(E_{r-1}) \).
  Thus \( \phi = \connection \psi = 0 \in \Gamma(\Wedge^1 \otimes E_{r} / E_{r-1}) \).
\end{proof}

Lemma \ref{BensLemma1} has the following consequence, eliciting a case where it is a purely formal calculation to see that \( \connection \) is exact.
We will, in fact, apply this criterion to the Killing operator on Cahen--Wallach spaces below.
\begin{proposition}\label{EeqEk}
  Assume that the curvature filtration is parallel for \( \connection \) and that $E_R=E$. Then \( \connection \) is exact.
\end{proposition}
\begin{proof}
  A given element of \( H^1(E,\connection) \) may be represented by a section $\phi \in \Gamma(\Wedge^1\otimes E)$ for which \( \dconn \phi = 0 \).
  Since \( \phi \in \Gamma(\Wedge^1 \otimes E_R) \) we may proceed by induction using Lemma~\ref{BensLemma1} to conclude that \( \phi \in \Gamma(\Wedge^1 \otimes E_0) \).
  But \( \connection \) is flat on \( E_0 \), and \( \dconn \phi = 0 \), so there is \( \eta \in \Gamma(E_0) \) for which \( \phi = \connection \eta \).
\end{proof}

It is now straighforward to show the following.
\begin{theorem}\label{thm: prop E/E_r exact implies E exact}
  If the curvature filtration \( (E_r) \) of a connection \( \connection \) is parallel and the induced connection on \( E / E_R \) is exact, then \( D \) is exact on \( E \).
\end{theorem}
\begin{proof}
  \( E / E_R \) is has injective curvature  by construction, so the result follows from
  Propositions \ref{thm: sub-quotient exactness} and \ref{EeqEk}.
\end{proof}

We will also need exactness in the following situation.
\begin{corollary}\label{corollary-complement}
  If the kernel of curvature \( E_0 \) is parallel and $E$ admits a $\connection$-parallel complement $ C$ to $E_0$, then $E_0=E_1$.
  In this case, $(E,\connection)$ is exact if and only if the induced connection on $ C$ is exact. 
\end{corollary}
\begin{proof}
  That \( E_0 = E_1 \) follows from the assumption that \( C \) is preserved by \( \connection \) and hence by \( \kappa \):
  given an element \( \phi \in \Gamma(E_1) \), decompose as \( \phi = \phi' + \phi'' \in \Gamma(E_0 \oplus C) = \Gamma(E) \).
  One finds that
  \[ \kappa(\phi'') = \kappa(\phi'+\phi'') = \kappa(\phi) \in \Gamma(\Wedge^2 \otimes C) \cap \Gamma(\Wedge^2 \otimes E_0) = 0 , \]
 so that \( \phi'' \in \Gamma(E_0 \cap C) = 0 \) and thus \( \phi \in \Gamma(E_0) \).
  The second statement follows from the consequent identification of \( C \) with \( E / E_0 \).
\end{proof}

\begin{remark}
We remark without expanding on the proof that Theorem \ref{thm: prop E/E_r exact implies E exact} is subsumed by a spectral sequence argument, which also generalises the result.
Indeed, supposing that the curvature filtration is parallel, there is a natural induced filtration on the quotient complex \eqref{eq: gauge complex}.
From the induced filtration spectral sequence, one finds that \( H^1(E, \connection) \subseteq H^1(E/E_R, \connection) \).
In particular, if \( H^1(E/E_R, \connection) = 0 \), then \( H^1(E,\connection) = 0 \).
\end{remark}

The following Lemma provides an effective characterization for connections with parallel curvature filtration, dependent only on curvature and its first derivatives.
\begin{lemma}\label{dkappa-lemma}
  For a connection \( \connection \) on \( E \) with curvature \( \kappa \), the following are equivalent:
  \begin{enumerate}
    \item The curvature filtration of \( \connection \) is parallel.

    \item For each \( r = 0, \ldots , R \), the restriction to \( E_r \) of augmented curvature \( E \to \Delta^2 \otimes E  \) takes values in \( \Delta^2 \otimes E_{r-1} \) .

    \item For any choice of torsion-free affine connection, and for each \( r = 0, \ldots, R \),
    \begin{equation*}
      \phi^\alpha \in \Gamma(E_r) \quad \mbox{implies that} \quad \bigl(\connection_a \kappa_{bc}{}^\beta{}_\alpha\bigr)\phi^\alpha \in \Gamma(\Wedge^1 \otimes \Wedge^2 \otimes E_{r-1})  .
    \end{equation*}

  \end{enumerate}
\end{lemma}
The statement of condition \( (3) \) depends on a choice of affine connection in order to extend \( \connection \) to a connection on \( \Wedge^2 \otimes E \), which is again denoted \( \connection \).
However, the 
validity of the statement does not depend on the choice.
This follows from the equivalence with the other two conditions, or is simple to check directly.
Note as well that the tensor \( (\connection_a\kappa_{bc}{}^\beta{}_\alpha)\phi^\alpha \) has vanishing skew-symmetric component, so condition \( (3) \) is a statement on the remaining hook-symmetric component.
\begin{proof}
  Consider first the restriction to \( E_0 \): for a choice of torsion-free affine connection, and any \( \phi^\alpha \in \Gamma(E_0) \),
  it follows from \( 0 = \tensor{\kappa}{_b_c^\beta_\alpha}\phi^\alpha \) that
  \[ 0 = \connection_a \bigl(\kappa_{bc}{}^\beta{}_\alpha \phi^\alpha\bigr) =
  \bigl(\connection_a \kappa_{bc}{}^\beta{}_\alpha \bigr)\phi^\alpha + \kappa_{bc}{}^\beta{}_\alpha (\connection_a \phi^\alpha) .
  \]
  As such, one has
  \( (\connection_a \kappa_{bc}{}^\beta{}_\alpha) \phi^\alpha = 0 \) if and only if  \( \kappa_{bc}{}^\beta{}_\alpha (\connection_a \phi^\alpha) = 0 \), in turn if and only if \( \connection_a \phi^\alpha \in \Gamma(\Wedge^1 \otimes E_0) \).
  Since \( E_0 \) is parallel if and only if \( \connection_a \phi^\alpha \in \Gamma(\Wedge^1 \otimes E_0) \) for all sections \( \phi^\alpha \) of \( E_0 \), this shows the equivalence of (1) and (3) on \( E_0 \).
  The equivalence of (3) and (2) follows from the formula for augmented curvature (Proposition \ref{thm: prop augmented curvature formula}).

  The general statement follows inductively.
  Suppose that the curvature filtration is parallel at least up to \( E_{r} \).
  The quotient of \( E_{r+1} \) by \( E_r \) reduces to the flat case, because
  \begin{itemize}
    \item \( E_{r+1} \) is parallel if and only if \( E_{r+1} / E_r \) is parallel in \( E / E_r \).

    \item \( \mathcal{D}\connection({E_{r+1}}) \subseteq \Delta^2 \otimes E_r \) if and only if \( \mathcal{D}\connection(E_{r+1} / E_r) = 0 \) in \( E / E_r \).

    \item \( (D \kappa) (E_{r+1}) \subseteq \Wedge^1 \otimes \Wedge^2 \otimes E_r \) if and only if \( (D \kappa) (E_{r+1}/E_r) = 0 \) in \( E / E_r \).
  \end{itemize}
  If any of \( (1), (2), \) or \( (3) \) hold for all \( r \), then the others do as well.
  If any of \( (1), (2), \) or \( (3) \) fails for some (minimal) \( r \), then the others do as well.
\end{proof}

\section{The Killing operator and the Killing connection}
\label{kill_and_more_kill}
Recall the diagram~(\ref{shiny_new_diagram}),
\begin{equation*}\begin{array}{cccccccccccc}
0 & \to & \Wedge^2&\to&
\begin{array}{c}\Wedge^2\\[-3pt]\oplus\\[-2pt]
\Wedge^2\otimes\Wedge^1\end{array}&\to&
\begin{array}{c}\Wedge^2\otimes\Wedge^1\\[-3pt]\oplus\\[-2pt]
\Wedge^3\otimes\Wedge^1\end{array}&\to&
\begin{array}{c}\Wedge^3\otimes\Wedge^1\\[-3pt]\oplus\\[-2pt]
\Wedge^4\otimes\Wedge^1\end{array}&\to&\cdots\\
&& \uparrow&&\uparrow&&\uparrow&&\uparrow\\
&& E&\stackrel{D}{\longrightarrow}&\Wedge^1\otimes E&
\stackrel{\dconn}{\longrightarrow}&\Wedge^2\otimes E&
\stackrel{\dconn}{\longrightarrow}&\Wedge^3\otimes E&
\stackrel{\dconn}{\longrightarrow}&\cdots\\
&& \uparrow&&\uparrow&&\uparrow&&\uparrow\\
&& \Wedge^1&\stackrel{\mathcal{K}}{\longrightarrow}&
\bigodot^2\!\Wedge^1&\stackrel{\mathcal{C}/2}{\longrightarrow}&
\begin{picture}(12,6)
\put(0,-6){\line(1,0){12}}
\put(0,0){\line(1,0){12}}
\put(0,6){\line(1,0){12}}
\put(0,-6){\line(0,1){12}}
\put(6,-6){\line(0,1){12}}
\put(12,-6){\line(0,1){12}}
\end{picture}&\stackrel{\mathcal{B}}{\longrightarrow}&
\begin{picture}(12,6)
\put(0,-12){\line(1,0){6}}
\put(0,-6){\line(1,0){12}}
\put(0,0){\line(1,0){12}}
\put(0,6){\line(1,0){12}}
\put(0,-12){\line(0,1){18}}
\put(6,-12){\line(0,1){18}}
\put(12,-6){\line(0,1){12}}
\end{picture}&\longrightarrow&\cdots .
\\[2mm]
\end{array}\end{equation*}
Our main task in this section is to finish its construction and justify the claims of exactness and commutativity made in the introduction, thus completing the proof of
Theorem~\ref{shiny_new_theorem}.
The missing formul{\ae} are as follows.
\begin{enumerate}[(a)]
\item The first two horizontal operators along the top row are given by
$$\lambda_{de}\mapsto
\left[\!\begin{array}{cc}-\lambda_{cd}\\
\nabla_{[c}\lambda_{d]e}\end{array}\!\right]
\quad\mbox{and}\quad
\left[\!\begin{array}{cc}\sigma_{cd}\\ \lambda_{cde}\end{array}\!\right]
\mapsto
\left[\!\begin{array}{cc}\nabla_{[b}\sigma_{c]d}+\lambda_{bcd}\\
\nabla_{[b}\lambda_{cd]e}+R_{e[b}{}^f{}_c\sigma_{d]f}
\end{array}\!\right]$$
followed by
$$\left[\!\begin{array}{cc}\sigma_{bcd}\\ \lambda_{bcde}\end{array}\!\right]
\mapsto
\left[\!\begin{array}{cc}\nabla_{[a}\sigma_{bc]d}-\lambda_{abcd}\\
\nabla_{[a}\lambda_{bcd]e}-R_{e[a}{}^f{}_b\sigma_{cd]f}\end{array}\!\right],
\quad\mbox{and so on.}$$
\item\label{vert1} The first two vertical sequences are given by
$$\begin{array}{cc}
&\makebox[0pt]{$\mu_{de}-\nabla_{[d}\sigma_{e]}$}\\ &\uparrow\\
\left[\!\begin{array}{c}\sigma_d\\
\nabla_{[d}\sigma_{e]}\end{array}\!\right]\!\!&\!\!
\left[\!\begin{array}{c}\sigma_d\\ \mu_{de}\end{array}\!\right]\\
\uparrow\\ \sigma_d
\end{array}\quad\mbox{and}\quad
\begin{array}{cc}
&\makebox[0pt]{$\left[\!\begin{array}{c}\sigma_{[cd]}\\
\mu_{[cd]e}+\nabla_{[c}h_{d]e}\end{array}\!\right]$,\enskip
\makebox[0pt][l]{where $h_{de}\equiv\sigma_{(de)}$}}\\
&\uparrow\\ \left[\!\begin{array}{c}h_{cd}\\
2\nabla_{[d}h_{e]c}\end{array}\!\right]\!\!&\!\!
\left[\!\begin{array}{c}\sigma_{cd}\\ \mu_{cde}\end{array}\!\right].\\
\uparrow\\ h_{cd}
\end{array}\hspace{130pt}$$
These are typical `splitting operators' from parabolic differential
geometry~\cite{HammerlSombergSoucekSilhan12}.
\item The remaining vertical sequences are given by
$$\begin{array}{cc}
&\left[\!\begin{array}{c}\sigma_{bcd}\\ \mu_{[bcd]e}\end{array}\!\right]\\
&\uparrow\\
\left[\!\begin{array}{c}0\\
\mu_{bcde}\end{array}\!\right]\!\!\!&\!\!\!
\left[\!\begin{array}{c}\sigma_{bcd}\\ \mu_{bcde}\end{array}\!\right]\\
\uparrow\\ \mu_{bcde}
\end{array}\quad\mbox{and}\quad\begin{array}{cc}
&\left[\!\begin{array}{c}\sigma_{abcd}\\ \mu_{[abcd]e}\end{array}\!\right]\\
&\uparrow\\
\left[\!\begin{array}{c}0\\
\mu_{abcde}\end{array}\!\right]\!\!\!&\!\!\!
\left[\!\begin{array}{c}\sigma_{abcd}\\ \mu_{abcde}\end{array}\!\right],\\
\uparrow\\ \mu_{abcde}
\end{array}\quad\mbox{and so on}\,.$$
\end{enumerate}
It is now straightforward to verify the key properties
of~(\ref{shiny_new_diagram}), namely that
\begin{itemize}
\item this diagram commutes,
\item the columns are short exact sequences,
\item the top row is an exact complex,
\end{itemize}
the verification of which is discussed in the following proof.

\begin{proof}[Proof of Theorem~\ref{shiny_new_theorem}]
  We restrict attention to the first three columns of the diagram, as that is what is needed for Theorem~\ref{shiny_new_theorem}.

  \subsubsection*{Exactness of the columns:}
  That the first column is exact is immediate from the defining formul{\ae}.
  Likewise, exactness of the third column is easy from the fact that the Bianchi identity characterises elements of \( \Wedge^{2} \otimes \Wedge^{2} \) with curvature tensor symmetries.
  Exactness of the middle column is a direct calculation, the key point being that the skewing map \( \Wedge^{1} \otimes \Wedge^{2} \to \Wedge^{2} \otimes \Wedge^{1} \) is an isomorphism.

  \subsubsection*{Exactness of the top row:}
  If an element \( \left[\!\begin{array}{cc}\sigma_{cd}\\ \lambda_{cde}\end{array}\!\right] \) is in the kernel of the second operator on the top row, then clearly \( \lambda_{cde} = - \nabla_{[c}\sigma_{d]e} \).
  This is precisely the condition to be in the image of the first operator.

  \subsubsection*{Commutativity:}
  The commutativity of the lower left square is equivalent to
  \[\nabla_b\sigma_c\!-\!\nabla_{[b}\sigma_{c]}\!=\!\nabla_{(b}\sigma_{c)}
  \ \text{ and }\ 
    \nabla_a\nabla_{[b}\sigma_{c]}-R_{bc}{}^e{}_a\sigma_e\!=\! \nabla_{[b}\nabla_{c]}\sigma_a\!+\!\tfrac12\left( \nabla_b\nabla_a\sigma_c\!-\!\nabla_c\nabla_a\sigma_b\right).
  \]
  The first equation holds trivially, and the second follows from the definition of the Riemannian curvature tensor    and its Bianchi symmetry, i.e.~from
  \[ R_{bc}{}^e{}_a\sigma_e=-2\nabla_{[b}\nabla_{c]}\sigma_a\quad\text{and}\quad
 R_{[ab}{}^e{}_{c]}=0.\]
 
  Commutativity of the lower right square follows from Equation (11)  in~\cite{CostanzaEastwoodLeistner21}, obtained from a lengthy but straightforward calculation.
  By mapping \( h_{bc} \) vertically and then applying \( \dconn \), one has
  \[ \left[ \!\begin{array}{c}
    0 \\
    \nabla_a\nabla_{[c}h_{d]b}- \nabla_b\nabla_{[c}h_{d]a}-R_{cd}{}^e{}_{[a}h_{b]e}
  \end{array}\! \right] =
\left[ \!\begin{array}{c}
  0 \\
  \tfrac{1}{2}\mathcal{C}(h)_{abcd}
\end{array}\! \right]
, \]
which is  (12)  in~\cite{CostanzaEastwoodLeistner21}.
The equality can be established by commuting derivatives and using Bianchi symmetry.

Finally, 
commutativity of the top left and right squares follows from the lower commutativity and vertical exactness.
Indeed, the operators on the top row can be defined by choosing any lift to the middle row, applying \( D \) or \( \dconn \), and projecting back to the top row.

\medskip
We conclude the  proof of the theorem by recalling that the homomorphism $\mathcal R:E\to \begin{picture}(12,12)(0,2)
\put(0,0){\line(1,0){12}}
\put(0,6){\line(1,0){12}}
\put(0,12){\line(1,0){12}}
\put(0,0){\line(0,1){12}}
\put(6,0){\line(0,1){12}}
\put(12,0){\line(0,1){12}}
\end{picture}$
in~(\ref{fancyR})  originates from the curvature $\kappa$ of the Killing connection in~(\ref{Killing_curvature}). Under the assumption that the rank of $\kappa$ is constant, we have  
the quotient of vector bundles
\[
\overline{\begin{picture}(12,12)(0,2)
\put(0,0){\line(1,0){12}}
\put(0,6){\line(1,0){12}}
\put(0,12){\line(1,0){12}}
\put(0,0){\line(0,1){12}}
\put(6,0){\line(0,1){12}}
\put(12,0){\line(0,1){12}}
\end{picture}}
\equiv
\displaystyle
\begin{picture}(12,12)(0,2)
\put(0,0){\line(1,0){12}}
\put(0,6){\line(1,0){12}}
\put(0,12){\line(1,0){12}}
\put(0,0){\line(0,1){12}}
\put(6,0){\line(0,1){12}}
\put(12,0){\line(0,1){12}}
\end{picture}\,/{{\mathcal{R}}(E)}.
\]
Then, 
 in the third column of~\eqref{shiny_new_diagram} we may replace \( \Wedge^{2} \otimes E \) with its quotient by \( \kappa(E) \) and 
 \( \begin{picture}(12,12)(0,2)
\put(0,0){\line(1,0){12}}
\put(0,6){\line(1,0){12}}
\put(0,12){\line(1,0){12}}
\put(0,0){\line(0,1){12}}
\put(6,0){\line(0,1){12}}
\put(12,0){\line(0,1){12}}
\end{picture} \) by its quotient 
\( \overline{\begin{picture}(12,12)(0,2)
\put(0,0){\line(1,0){12}}
\put(0,6){\line(1,0){12}}
\put(0,12){\line(1,0){12}}
\put(0,0){\line(0,1){12}}
\put(6,0){\line(0,1){12}}
\put(12,0){\line(0,1){12}}
\end{picture}} \,, \) 
and the resulting column will remain exact. With the exactness of the top row and the commutativity of the diagram, a typical diagram chase implies the desired equivalence in Theorem~\ref{shiny_new_theorem}.
\end{proof}

In fact, the connection (\ref{killing_connection}) and the commutative diagram
(\ref{shiny_new_diagram}) with its very useful properties, as above, are
available for any torsion-free affine connection (and, although we shall not
need it, these constructions are {\em projectively invariant\/}).  More
generally, we may consider the operator
$$\sigma_a\longmapsto\nabla_{(a}\sigma_{b)},$$
where $\sigma_a$ is a $1$-form on an arbitrary smooth manifold and $\nabla_a$
is an arbitrary torsion-free affine connection.  We call the $1$-forms in the
kernel of this operator \emph{Killing covectors\/} and (essentially by design)
we can identify them with $D_a$-parallel sections of the bundle~$E$.
Properties of this {\em Killing operator\/} can generally be read off from the
diagram~(\ref{shiny_new_diagram}).  In particular, both its kernel and the
range are determined by the kernel and range of the Killing connection, as in
the following theorem, the first part of which is well-known (essentially since
\cite{Kostant55}).
\begin{theorem}\label{killtheo}
Let $\nabla$ be a torsion free affine connection on a smooth manifold and let
$\D$ be the Killing connection on the bundle $E\equiv\Wedge^1\oplus\Wedge^2$.
\begin{enumerate}
\item The map
\[\Wedge^1\ni\sigma_b\longmapsto
\left[\!\begin{array}{c}\sigma_b\\ \nabla_{[b}\sigma_{c}]\end{array}\!\right]
\in E\]
is an isomorphism between Killing covectors and parallel sections of $\D$.
\item\label{range_condition} Regarding the range of the Killing operator,
$$h_{ab}=\nabla_{(a}\sigma_{b)}\iff
\left[\!\begin{array}{c}h_{ab}\\ 2\nabla_{[b}h_{c]a}
\end{array}\!\right]
=D_a\left[\!\begin{array}{c}\sigma_b\\
\mu_{bc}\end{array}\!\right],
\enskip\mbox{for some }\mu_{bc}\in\Wedge^2.$$
\end{enumerate}
\end{theorem}
As regards this article, however, the main yield from
diagram~(\ref{shiny_new_diagram}) is Theorem~\ref{shiny_new_theorem}, which,
in the terminology of Definition~\ref{exact_connection} and applied to locally symmetric spaces, reads as follows. 
\begin{corollary}\label{2exact}
Let $ (M,g)$ be a semi-Rieman\-nian locally symmetric space.  Then the complex
(\ref{key_complex}) is exact if and only if the Killing connection is exact.
\end{corollary}
Exactness of the Killing connection, however, is quite a difficult issue and
will occupy the rest of this article, finally leading to our proof of
Theorem~\ref{calabi-theo} via Theorem~\ref{killcon-exact} and this
Corollary~\ref{2exact}.  In particular, even for locally symmetric spaces, it
is not always the case that the Killing connection is exact, and, in accordance
with Theorem~\ref{calabi-theo}, it is already false for $S^1\times S^2$ (see
Example~\ref{coro:non_exactness}).

Instead, as indicated in \S\ref{subsection_augmented}, we can strengthen the
criteria for the range of the Killing connection by using augmented curvature.
Specifically, for a locally symmetric metric (of any signature) it is easy to
check, as observed in~\cite{CostanzaEastwoodLeistner21}, that the subbundle $E_0\subset E$ is
parallel for the Killing connection.  (Indeed, we shall soon show in
Proposition~\ref{Ekparallel} that the whole curvature filtration $E_k$ is
parallel.)  Therefore, Proposition~\ref{complete_characterization} applies and
so (\ref{eq-augmented curvature necessary condition}) gives precise criteria
for a section of $\Wedge^1\otimes E$ to be in the range of the Killing
connection, the first component of which leads exactly to
(\ref{curvature_reduced_complex}) and then to (\ref{key_complex}) and
Theorem~\ref{calabi-theo}.  We obtain a stronger (second order) criterion from
the second component of~(\ref{eq-augmented curvature necessary condition}) and
it is simply a matter of substituting formula (\ref{range_condition}) of
Theorem~\ref{killtheo} to unpack this criterion as a third order differential
condition on $h_{ab}\in\Gamma(\bigodot^2\!\Lambda^1)$ in order that it be
locally in the range of the Killing operator.  Writing out this operator in
detail gives~\cite[Th\'eor\`eme~7.2]{GasquiGoldschmidt83}.

\section{The Killing connection for semi-Rieman\-nian locally symmetric spaces}
\label{sec:products}
In this section we  study the exactness of the Killing connection for semi-Rieman\-nian symmetric spaces in order to provide the ingredients for the proof of Theorem \ref{calabi-theo}, which will be carried out in Section~\ref{lorsec}.
We will denote by $\nabla$  the Levi-Civita connection of a semi-Rieman\-nian manifold $(M,g)$ and $R=2\, \nabla^\w\circ \nabla$ its curvature tensor.

\subsection{The curvature filtration for the Killing connection}
In this section we divert from our usual notation and   write $R(X,Y)$ for $X,Y\in TM$, when we  consider the curvature tensor  $R$ as a section of $\Wedge^2\otimes \so(TM,g)$, so that $R(X,Y)\in \Gamma\left( \so(TM,g)\right)$ denotes $R$ applied to $X$ and $Y$.  With the convention for $R_{abcd}$ from the introduction it is
\[\left(R(X,Y)Z\right))^d=-X^aY^bZ^c R_{abc}{}^d.\]
Since $(M,g)$ is locally symmetric, i.e.~$\nabla R=0$, in the following we can work at a fixed point, which we will omit from notation. Moreover, local symmetry implies that
\[\hol=\mathrm{span}\{ R(X,Y)\mid X,Y\in TM\}\]
is the holonomy algebra of $(M,g)$. We define
\[
\aut(R):=\{A\in \so(TM,g)\mid A\cdot R=0\}.
\]
Here  $A\cdot R$ denotes the action of an endomorphism $A\in \so(TM,g)$ on the curvature tensor $R\in
\Lambda^2\otimes \so(TM,g)$,
\begin{equation}\label{action}
 (A\cdot R)(X,Y) =
\left[ A,R(X,Y)\right] -R(AX,Y)-R(X,AY),\end{equation}
where $[-,-]$ denotes the commutator of two endomorphisms in $\so(TM)$.
In indices, with $A^e_{~f}\in \so(TM)$ and $R_{ab}{}^c{}_d\in \Lambda^2\otimes \so(TM,g)$, this is
\[
(A\cdot R)_{ab}{}^c{}_d
=
A^c_{~e}R_{ab}{}^e{}_d
-
A^e_{~d}R_{ab}{}^c{}_e
-
A^e_{~a}R_{eb}{}^c{}_d
-
A^e_{~b}R_{ae}{}^c{}_d,
\]
so that
\[
(A\cdot R)_{abcd}
=
-2\left( R_{ab}{}^e{}_{[c}A_{d]e}+
R_{cd}{}^e{}_{[a}A_{b]e}\right).
\]
We fix the vector bundle $E$ with the Killing connection $D$, as before,
\begin{equation}
E:=\ostack{TM}{\so(TM,g)},\qquad
\label{connect}\D_X\begin{bmatrix} \xi\\A\end{bmatrix}=\begin{bmatrix} \nabla_X \xi-A(X)\\\nabla_XA+R(\xi,X)\end{bmatrix},\end{equation}
with curvature acting as
\begin{equation}\label{prolcurv}
\kappa \begin{bmatrix} \xi\\A\end{bmatrix}
=
\begin{bmatrix} 0\\-(A\cdot R)\end{bmatrix} \in \Lambda^2\otimes E.
\end{equation}
In order to describe the curvature filtration of Section \ref{sec-curvature filtration} for the Killing connection, we set $\h_0:=\aut(R)\subseteq \so(TM,g)$ and define inductively for $k>0$
\begin{equation}\label{hkdef}
\h_k:=\{A\in \so(TM,g)\mid [A,H]\in \h_{k-1} \text{ for all }H\in \hol\}\ .\end{equation}

\begin{lemma}
If $(M,g)$ is locally symmetric, then  $\hol\subseteq \h_0\subseteq ...\subseteq \h_k \subseteq \h_{k+1} \subseteq \ldots \subseteq \so(TM,g)$ is a   filtration of subalgebras, and the curvature filtration $E_k$ of the Killing connection is given by
\[E_k=\ostack{TM}{\h_k}.\]
\end{lemma}
\begin{proof}
For every locally symmetric space, since $\nabla R=0$, we have that $\hol\subseteq \h_0=\aut(R)$. That the $\h_k$ are subalgebras follows from the Jacobi identity.
The statement for  $E_0$  follows immediately from (\ref{prolcurv}). For $k>0$ we may assume  by induction  that
\[
 \begin{bmatrix} \xi\\A\end{bmatrix}\in E_k\iff  \begin{bmatrix}0 \\A\cdot R\end{bmatrix}\in \Lambda^2\otimes  E_{k-1}\iff
 A\cdot R\in \Lambda^2\otimes\h_{k-1}.\]
However, by (\ref{action}) we have that $ (A\cdot R)(X,Y) \equiv [A,R(X,Y)] + \hol$, which proves the statement of the lemma.
\end{proof}

We can also show that for locally symmetric spaces the curvature filtration is parallel.

\begin{proposition}\label{Ekparallel}
Let $(M,g)$ be a locally symmetric semi-Rieman\-nian manifold and $\overline{\D}$   the connection induced by the Killing connection and the Levi-Civita connection on $\Wedge^2\otimes E$.
Then
then $\overline{\D}\kappa\in \Gamma\left(\Wedge^1\otimes\Wedge^2 \otimes \mathrm{Hom}(E,E)\right)$ is given by
\begin{equation}\label{dkappa}
\overline{\D}_X\kappa(Y,Z) \begin{bmatrix} \xi\\A\end{bmatrix}=  \begin{bmatrix} (A\cdot R)(Y,Z)X\\ 0\end{bmatrix}.\end{equation}
In particular,  $\overline{\D}\kappa \in \Gamma\left( \Wedge^1\otimes \Wedge^2\otimes \mathrm{Hom}(E,TM)\right)$ and, by Lemma \ref{dkappa-lemma}, the curvature filtration $E_k$, $k\ge 0$, is parallel for $\D$.
\end{proposition}
\begin{proof}
The proof of equation (\ref{dkappa}) is a  direct computation: by the definition of the induced connection,
we have that
\begin{eqnarray*}
\lefteqn{
\overline{D}_X\kappa(Y,Z) \begin{bmatrix} \xi\\A\end{bmatrix}=}
\\
&=&
D_X\left(  \kappa(Y,Z) \begin{bmatrix} \xi\\A\end{bmatrix} \right)
-
\kappa(\nabla_XY,Z) \begin{bmatrix} \xi\\A\end{bmatrix}
-
\kappa(Y,\nabla_XZ) \begin{bmatrix} \xi\\A\end{bmatrix}
-
\kappa(Y,Z) D_X\begin{bmatrix} \xi\\A\end{bmatrix}
\\
&=&
 \begin{bmatrix} (A\cdot R) (X,Y)Z
  \\
  -\left(\nabla_X (A\cdot R)\right) (Y,Z) + (\nabla_XA+R(\xi,X))\cdot R(Y,Z) \end{bmatrix}.
\end{eqnarray*}
By local symmetry, $R(\xi,Z) \in \aut(R)$ and thus $R(\xi,Z) \cdot R=0$. For the remaining terms we apply
the Leibniz rule to get $\nabla_X (A\cdot R)=(\nabla_XA) \cdot R+A\cdot (\nabla_XR)$. Then local symmetry, $\nabla_XR=0$, gives  the desired equation  (\ref{dkappa}).

By Lemma \ref{dkappa-lemma}, equation (\ref{dkappa})  implies that the curvature filtration is parallel.
\end{proof}

We conclude this subsection
by
considering the situation when  $\h_0\subseteq \so(TM,g)$ is non-degenerate
with respect to the trace form $B$ in $\so(TM,g)$, so that
\[\so(TM,g)=\h_0\+\h_0^\perp.\]
This holds in particular when $g$ is Rieman\-nian.
Note that the $\mathrm{ad}$-invariance of $B$ immediately implies that
\begin{equation}\label{reductive}
 \left[\h_0,\h_0^\perp\right]\subseteq \h_0^\perp.\end{equation}
Considering the natural embedding $\h_0^\perp\subseteq E$, we have the   decomposition
\[E=E_0\+\h_0^\perp,\]
where $E_0=TM\+\h_0$ is a $D$-invariant and flat sub-bundle and $\kappa$ is injective on $\h_0^\perp$.
Note however that $\h_0^\perp\subseteq E$ is not $D$-parallel. Nevertheless we have the following (as in Corollary \ref{corollary-complement}).
\begin{proposition}\label{complement-prop}

Assume that $(M,g)$ is a semi-Rieman\-nian locally symmetric space such that $\h_0=\aut(R)$ satisfies
$\so(TM,g)=\h_0\+\h_0^\perp$.
 Then $E_0=E_1$. In addition,  $D$ is exact if  the connection ${\nabla}$ that is induced   on  $\h_0^\perp$ from the Levi-Civita connection  of $g$ is exact. \end{proposition}
\begin{proof}
If  $\eta =\eta_0+ C \in E_1$ with $\eta_0\in E_0=TM\+\h_0$ and $C\in \h_0^\perp$, then
\[
\kappa(\eta)
=
\kappa \left(\begin{bmatrix} 0\\C\end{bmatrix}\right)
=
\begin{bmatrix} 0\\C\cdot R\end{bmatrix}
 \in \Wedge^2 \otimes E_0.\]
 Hence, $C\cdot R\in \Wedge^2 \otimes \h_0$, which, by local symmetry and (\ref{action}), implies that
 $\left[ C,R(X,Y)\right]\in \h_0$ for all $X,Y\in TM$. With $R(X,Y)\in \h_0$ and (\ref{reductive}), this implies that
$ \left[ C,R(X,Y)\right]\in \h_0\cap\h_0^\perp$. Therefore, the assumption implies that
 $\left[ C,R(X,Y)\right]=0$ for all $X,Y\in TM$. By the pairwise symmetry of $R$, this  also implies that for all $X,Y\in TM$,
 \[R(CX,Y)+R(X,CY)=0.\]
Consequently, $C\cdot R=0$, so $C\in \h_0$, and thus $E_1=E_0$.

From Theorem \ref{thm: prop E/E_r exact implies E exact} it follows that $D$ on $E$ is exact if the connection that is induced on $E/E_0=\h_0^\perp$ is exact. However, for $C\in \h_0^\perp$ we have that
\[
\D_X\begin{bmatrix} 0\\C\end{bmatrix}=\begin{bmatrix}-C(X)\\\nabla_XC\end{bmatrix} \equiv \begin{bmatrix}0\\\nabla_XC\end{bmatrix} \mod{\quad TM\+\h_0}.\]
Since $\h_0^\perp$ is taken with respect to the metric on $\so(TM)$ induced from $g$, which is parallel with respect to $\nabla$, we get that $\nabla_XC$ is a section of $\h_0^\perp$, so the induced connection on $\h_0^\perp$ is indeed the Levi-Civita connection.
\end{proof}

\begin{remark}\label{riemremark}
In \cite{CostanzaEastwoodLeistner21} we have shown that the Levi-Civita connection on $\h_0^\perp$
is exact whenever $(M,g)$ is a Rieman\-nian locally symmetric space that is
irreducible (i.e.~indecomposable) or decomposable but without Hermitian and
flat factor.
More specifically, when $(M,g)$ is irreducible, the curvature
filtration already stabilises at $E_0$, the statement that the Levi-Civita
connection is exact on $\h_0^\perp$ is that Equation~(15) from \cite{CostanzaEastwoodLeistner21}
forces $X_{bcd}=0$, and this follows from Lemma~1 of \cite{CostanzaEastwoodLeistner21} by means
of Equation~(16), which is the trace of Equation~(15).
\end{remark}

\subsection{Products of locally symmetric spaces}
In this section we consider exactness  of the Killing connection for a product of two locally symmetric semi-Riemannian spaces, provided we have some information on  one of the spaces.
Our notation will be as follows: let $g+\bar g $ be a product metric on the product manifold $M\times \bar M$ of locally symmetric symmetric semi-Riemannian spaces $(M,g)$ and $(\bar M,\bar g)$ and let  $\Wedge^1$, $\Wedge^{1,0}$ and $\Wedge^{0,1}$ the respective  bundles of one-forms. We do not distinguish between bundles on $M$ and $\bar M$ and their pull-backs to $M\times \bar M$, so that on the product we have $\Wedge^1=\Wedge^{1,0}\+\Wedge^{0,1}$, and similarly for the two-forms $\Wedge^2=\Wedge^{2,0}\+\Wedge^{1,1}\+\Wedge^{0,2}$.
Then  we have that the curvature tensor of the product, which we denote by $R$, abusing notation,  satisfies
\begin{equation}
\label{nomixedR}
R\in \Gamma\left( (\Wedge^{2,0}\otimes \Wedge^{2,0}) \+(\Wedge^{0,2}\otimes \Wedge^{0,2})\right)
.\end{equation}

We denote the Killing bundles of $(M,g)$ and $(\bM,\bg)$ by $E$ and $\bE$. Again, by abuse of notation, we denote by $D$ the Killing connection of $(M\times \bM,g+\bg)$ on $\Wedge^1\+\Wedge^2$.
 Because $\Wedge^2$ on $M\times \bM$ decomposes as
$\Wedge^2=\Wedge^{2,0}\+\Wedge^{0,2}\+\Wedge^{1,1}$, the Killing bundle of $(M\times \bM, g+\bg)$ decomposes as
\[\begin{array}{c}
\Wedge^1\\
\+\\
\Wedge^2\end{array}=E\+\bE\+\Wedge^{1,1}=
\begin{array}{ccccc}
\Wedge^{1,0}&&\Wedge^{0,1}&\\
\+&\+&\+&&\\
\Wedge^{2,0}&&\Wedge^{0,2}&\+&\Wedge^{1,1}.
\end{array}.
\]
Since the curvature tensor satisfies (\ref{nomixedR}),  it is immediate from the definition of the Killing connection that $E $ and $ \bE$ are  parallel sub-bundles of the Killing bundle of the product.
Hence, $E\+\bE$ is also parallel and
 Proposition~\ref{prop:pullback_exactness} implies that $D|_{E\+\bE}$ is an exact connection on $E\+\bE$ 
 if the original connections on $E$ and $ \bE$ are exact. This allows to apply Corollary~\ref{coro:ker_D_exact} to obtain the following important criterion.

		\begin{proposition}\label{prop:products_exactness_aid}
			Let $(M, g)$ and $(\bM, \bg)$ be semi-Rieman\-nian locally symmetric spaces such that their Killing connections are exact. If the Killing connection $D$  of $(M\times\bM,g+\bg)$ satisfies  $\ker(D^{\wedge}) \subseteq
			\Gamma\left(\Wedge^1\otimes (E\+\bE)\right)$,
			 then  $D$ is exact.
		\end{proposition}

		We will now determine in which cases a product of locally symmetric spaces $(M\times \bM, g+\bg)$ with exact Killing connections satisfy $\ker(D^{\wedge}) \subseteq \Gamma\left( \Wedge^1\otimes (E\+\bE)\right)
$. We return to  use Penrose's abstract indices and we denote indices from $\Wedge^{1,0}$ with $a,b,\ldots $, indices from $\Wedge^{0,1}$ with barred indices $\bar{a}, \bar{b},\ldots$ and capital indices $A,B,\ldots $ for both groups of indices. We denote by $\nabla_A$ the Levi-Civita connection of the product metric and by $R_{ABCD}$ its curvature tensor.
 In this notation, if $D$ is the Killing connection of the product, \[D^{\wedge} : \Wedge^1 \otimes (\Wedge^1\+\Wedge^2)) \to \Gamma(\Wedge^{2} \otimes (\Wedge^1\+\Wedge^2),\] takes the form
		\begin{equation}\label{eqn:D_wedge}
		D^{\wedge}_{A}
		\begin{bmatrix}
			\alpha_{BC} \\
			\psi_{BCD}
		\end{bmatrix}
		=
		\begin{bmatrix}
			\nabla_{[A} \alpha_{B]C} + \psi_{[AB]C}\\
			\nabla_{[A} \psi_{B]CD} - R_{CD}{}^E{}_{[A} \alpha_{B]E}
		\end{bmatrix}.
		\end{equation}
		Notice that for $(\alpha, \psi) \in \ker(D^{\wedge})$, since the curvature tensor satisfies (\ref{nomixedR}), the second of the resulting equations in  (\ref{eqn:D_wedge}), with $ABCD = ABc\bar{d}$, yields the constraint
		\[
		\nabla_{[A}\psi_{B]c\bar{d}} = 0.
		\]
This means that the $\Wedge^1 \otimes \Wedge^{1,1}$ component of $\psi_{BCD}$ must be contained in the kernel of $\nabla^\w : \Wedge^1 \otimes \Wedge^{1,1} \to \Wedge^{2}  \otimes\Wedge^{1,1}$. The following lemma then shows that the curvature tensors of the factors cannot be injective.

		\begin{lemma}\label{prop:dnabla_kernel}
			On a product of semi-Rieman\-nian manifolds with curvature tensor $R_{ABCD}$, let $\phi_{BCD} \in \Gamma(\wedge^{1} \otimes \wedge^{1, 1})$ be a solution of
			\begin{equation}\label{eqn:nabla_wedge}
			\nabla_{[A}\phi_{B]CD} = 0.
			\end{equation}
			Then
			\begin{equation}\label{eqn:ker_nabla_wedge}
				R_{ab}{}^e{}_c \phi_{\bar{a} e \bar{b}} = 0, \quad \text{and} \quad R_{\bar{a}\bar{b}}{}^{\bar{e}}{}_{\bar{c}} \phi_{ab \bar{e}} = 0.
			\end{equation}
		\end{lemma}
		\begin{proof}
		Rewriting equation (\ref{eqn:nabla_wedge}) using barred and unbarred indices, we obtain
		\[
		\nabla_{[a} \phi_{b]c\bar{a}} = 0, \quad \nabla_{[\bar{a}} \phi_{\bar{b}]a\bar{c}} = 0, \quad  \nabla_{a} \phi_{\bar{a}b\bar{b}} - \nabla_{\bar{a}} \phi_{ab\bar{b}} = 0.
		\]
		Differentiating the first equation with respect to the barred indices, we have
		\[
			0  =  \nabla_{\bar{a}} (\nabla_{a} \phi_{bc\bar{b}} - \nabla_{b} \phi_{ac\bar{b}})
			    =  \nabla_{a} \nabla_{\bar{a}} \phi_{bc\bar{b}} - \nabla_{b} \nabla_{\bar{a}} \phi_{ac\bar{b}}.
		\]
		The third equation implies that
		\[
			 \nabla_{a} \nabla_{\bar{a}} \phi_{bc\bar{b}} - \nabla_{b} \nabla_{\bar{a}} \phi_{ac\bar{b}}=  \nabla_{a} \nabla_{b} \phi_{\bar{a}c\bar{b}} - \nabla_{b} \nabla_{a} \phi_{\bar{a}c\bar{b}}
			     =- R_{ab}{}^e{}_c \phi_{\bar{a}e\bar{b}},
		\]
		and therefore $R_{ab}{}^e{}_c \phi_{\bar{a}e\bar{b}} = 0$. Analogously, differentiating the second equation with respect to the unbarred indices, we obtain $R_{\bar{a}\bar{b}}{}^{\bar{e}}{}_{\bar{c}} \phi_{ab\bar{e}} = 0$, as claimed.
		\end{proof}
		Consequently, Lemma \ref{prop:dnabla_kernel} provides us with constraints for a section of $\Wedge^1 \otimes \Wedge^{1,1}$  to be in the kernel of $\nabla^\w$ when one of the factors has an injective curvature tensor.

		\begin{proposition}\label{prop:small_exactness}
			Let $(M,g)$ be a semi-Rieman\-nian locally symmetric space  that is either
	\begin{enumerate}
	\item  of non-zero constant sectional curvature, or
	\item non-Hermitian indecomposable Rieman\-nian.
\end{enumerate}
If $(\bM,\bg)$ is  another semi-Rieman\-nian locally symmetric space with exact Killing connection, then the Killing connection of the product $(M\times \bM, g+\bg)$ is exact.
		\end{proposition}
		\begin{proof}
		Let $(\alpha, \psi + \phi)$ be in the kernel of $D^{\wedge}$, with $(\alpha, \psi) \in \Gamma(\Wedge^1 \otimes (E \oplus \bE))$ and $\phi \in \Gamma(\Wedge^1 \otimes \Wedge^{1,1})$. We will show that $\phi = 0$ and consequently, the exactness of $D$ will follow from Proposition \ref{prop:products_exactness_aid}.

		The $\Wedge^{2} \otimes \Wedge^{2}$ component of equation (\ref{eqn:D_wedge}), with $CD = c\bar{d}$, becomes
		\[
		\nabla_{[A} \phi_{B]c\bar{d}} = 0,
		\]
		i.e. $\phi$ is in the kernel of $\nabla^{\wedge}$. By hypothesis, $(M,g)$ is semi-Rieman\-nian of non-zero constant sectional curvature or a non-Hermitian indecomposable Rieman\-nian symmetric space, for which $R_{abc}{}^d$ has trivial kernel as a homomorphism from $\Wedge^{1,0}$ to $ \Wedge^{2,0}\otimes\Wedge^{1,0}$.
		 Therefore, Lemma \ref{prop:dnabla_kernel} guarantees us that $\phi_{\bar{a}a\bar{b}} = 0$.

		Now we will show that $\phi_{ab\bar{a}} = \phi_{(ab)\bar{a}}$. Since $\nabla_{[b}\phi_{c]d\bar{e}} = 0$, we have
		\begin{equation}\label{eqn:genericity}
			0 = \nabla_{[a}\nabla_{b}\phi_{c]d\bar{e}} = - R^{e}_{\;\; d[ab}\phi_{c]e\bar{e}}.
		\end{equation}
		Fixing $X^{\bar{e}}$, set $X^{\bar{e}}\phi_{ab\bar{e}}=h_{ab} + \omega_{ab} $, with $h_{ab} = X^{\bar{e}}\phi_{(ab)\bar{e}}$ and $\omega_{ab} = X^{\bar{e}}\phi_{[ab]\bar{e}}$. Contracting equation (\ref{eqn:genericity}) with $X^{\bar{e}}$ we obtain
		\begin{equation}\label{eqn:genericity_2}
			0 = R^e{}_{d[ab} h_{c]e} + R^e{}_{d[ab} \omega_{c]e}.
		\end{equation}
		In the case of (1), i.e.~that $(M,g)$ is semi-Rieman\-nian of non-zero constant sectional curvature, we have that $R_{abcd}$ is a non-zero constant multiple of $ g_{ac}g_{bd}-g_{ad}g_{bc}$, so that  equation~(\ref{eqn:genericity_2}) gives
\[
0=
g_{e[a}h_{c}{}^eg_{b]d}
-
g_{e[b}h_{c}{}^eg_{a]d}
+
g_{e[a}\omega_{c}{}^eg_{b]d}
-
g_{e[b}\omega_{c}{}^eg_{a]d}
=
-2		g_{e[a}\omega_{b}{}^eg_{c]d},
\]
since $h_{ab}$ is symmetric.
But this is nothing else than
\[0=	 \omega_{[ab}g_{c]d},\]
which implies that $\omega_{ab}=0$.

In case (2), when $(M,g)$ is non-Hermitian indecomposable Rieman\-nian, tracing equation (\ref{eqn:genericity_2}) over $bd$ we have
		\begin{equation}\label{eqn:genericity_traced}
			0 = 2 R^{e}{}_{[a}h_{c]e} + 2 R^{e}{}_{[a} \omega_{c]e} + R_{ac}{}^{ef} \omega_{ef}.
		\end{equation}
		As $(M,g)$ is an indecomposable Rieman\-nian symmetric space, it is Einstein and without loss of generality, we can assume $R_{ab} = \pm  g_{ab}$. Then equation (\ref{eqn:genericity_traced}) becomes the eigenvalue equation
		\begin{equation}\label{eqn:eigenvalue}
R_{ac}{}^{ef}\omega_{ef}= \pm 2 \omega_{ac}
		\end{equation}
		for a 2-form on $M$.  It was proved in  \cite[Theorem 2]{CostanzaEastwoodLeistner21} that  if equation (\ref{eqn:eigenvalue}) holds 
		on an indecomposable Riemannian locally symmetric space $(M,g)$, then  $(M,g)$ is a Hermitian locally symmetric space  with  $\omega_{ab}$  a  multiple of its K\"ahler form. Therefore, by our assumption on $(M,g)$, we obtain as well  that $ \omega_{ab} = 0$.

		Hence, in both cases we have
		$X^{\bar{c}}\phi_{[ab]\bar{c}} = \omega_{ab} = 0$.
 This holds for every vector field $X^{\bar{a}}$ and therefore it must hold that $\phi_{[ab]\bar{a}} = 0$.

			From equation (\ref{eqn:D_wedge}) with $ABCD = abcd$, we know that
			\begin{equation}\label{eqn:alpha_psi_unbarred}
			\alpha_{ab} = \nabla_{a}\sigma_{b} - \mu_{ab} \quad \text{and} \quad \psi_{abc} = \nabla_{a}\mu_{bc} - R_{bc}{}^e{}{}_{a}\sigma_{e}
			\end{equation}
		by the exactness of the Killing connection of $(M,g)$, for some $(\sigma, \mu) \in \Gamma(E_{0})$. Notice that $\alpha_{ab}$ and $\psi_{abc}$  still depend on variables from $\bM$. From the $\Wedge^{2}\otimes \Wedge^1$ component in equation (\ref{eqn:D_wedge}), with $ABC = \bar{a}bc$, and equation (\ref{eqn:alpha_psi_unbarred}), we have
		\[
		\nabla_{\bar{a}}(\nabla_{b}\sigma_{c} - \mu_{bc}) - \nabla_{b} \alpha_{\bar{a}c} + \psi_{\bar{a}bc} + \phi_{bc\bar{a}} = 0.
		\]
		Defining $\theta_{\bar{a}b} := \alpha_{\bar{a}b} - \nabla_{\bar{a}}\sigma_{b}$, it takes the form
		\begin{equation}\label{eqn:alpha_psi_phi}
			- \nabla_{b}\theta_{\bar{a}c} - \nabla_{\bar{a}}\mu_{bc} + \psi_{\bar{a}bc} + \phi_{bc\bar{a}} = 0.
		\end{equation}
		Symmetrising and skew-symmetrising equation (\ref{eqn:alpha_psi_phi}) in $bc$, we obtain
		\begin{equation}
			\psi_{\bar{a}bc} = \nabla_{[b} \theta_{|\bar{a}|c]} + \nabla_{\bar{a}} \mu_{bc} \quad \text{and} \quad \phi_{bc\bar{a}} = \nabla_{(b} \theta_{|\bar{a}|c)},
		\end{equation}
		as $\psi_{\bar{a}bc} = \psi_{\bar{a}[bc]}$ and $\phi_{ab\bar{a}} = \phi_{(ab)\bar{a}}$. Then
		\begin{equation}\label{eqn:mixed_D_wedge_1}
			\begin{array}{rcl}
				\nabla_{\bar{a}} \psi_{bcd} - \nabla_{b}\psi_{\bar{a}cd} & = & \nabla_{\bar{a}}(\nabla_{b}\mu_{cd} - R_{cd}{}^e{}_b \sigma_{e}) - \nabla_{b}(\nabla_{[c} \theta_{|\bar{a}|d]} + \nabla_{\bar{a}} \mu_{cd}) \\
				& = & (\nabla_{\bar{a}}\nabla_{b} - \nabla_{b}\nabla_{\bar{a}}) \mu_{cd} - R_{cd}{}^e{}_b \nabla_{\bar{a}} \sigma_{e} - \nabla_{b} \nabla_{[c} \theta_{|\bar{a}|d]} \\
				& = & - R_{cd}{}^e{}_b \nabla_{\bar{a}} \sigma_{e} - \nabla_{b} \nabla_{[c} \theta_{|\bar{a}|d]} .
			\end{array}
		\end{equation}
		The $\Wedge^{2} \otimes \Wedge^{2}$ component of equation (\ref{eqn:D_wedge}) with $ABCD = \bar{a}bcd$ is
		\begin{equation}\label{eqn:mixed_D_wedge_2}
		\nabla_{\bar{a}} \psi_{bcd} - \nabla_{b}\psi_{\bar{a}cd} = - R_{cd}{}^e{}_b \alpha_{\bar{a}e}.
		\end{equation}
		Combining equations (\ref{eqn:mixed_D_wedge_1}) and (\ref{eqn:mixed_D_wedge_2}) we obtain
		\[
		R_{cd}{}^e{}_b (\alpha_{\bar{a}e} - \nabla_{\bar{a}} \sigma_{e}) = \nabla_{b} \nabla_{[c} \theta_{|\bar{a}|d]}.
		\]
		or
		\begin{equation}\label{eqn:the_eqn}
		R_{cd}{}^e{}_b \theta_{\bar{a}e} = \nabla_{b} \nabla_{[c} \theta_{|\bar{a}|d]}.
		\end{equation}
		Since $(M,g)$ is a space of non-zero constant sectional curvature or a non-Hermitian indecomposable Rieman\-nian symmetric space, by \cite[Lemma 3]{CostanzaEastwoodLeistner21}, equation (\ref{eqn:the_eqn}) implies that
		\begin{equation}
			\nabla_{b}\theta_{\bar{a}c} = \nabla_{[b}\theta_{|\bar{a}|c]}.
		\end{equation}
		This means that $\phi_{bc\bar{a}} = \nabla_{(b}\theta_{|\bar{a}|c)} = 0$ and therefore $\phi_{ABC} = 0$. We have shown that $\ker(D^{\wedge}) \subseteq \Gamma(\Wedge^1\otimes (E \oplus \bE))$ and therefore, since $D|_{E \oplus \bE}$ is exact, $D$ is exact by Corollary~\ref{coro:ker_D_exact}.
		\end{proof}

 Hermitian symmetric spaces have to be excluded in Proposition \ref{prop:small_exactness}, as it is  illustrated by the following example, a version of which  already appeared in \cite[Proposition 4]{CostanzaEastwoodLeistner21}.
		\begin{example}\label{coro:non_exactness}
			Let $(M,g)$ be an indecomposable Rieman\-nian Hermitian symmetric space and let $(\bM,\bg)$ be a symmetric space  with non-injective curvature  $R_{\bar a \bar b\bar c}{}^{\bar d}$, i.e.~with a parallel  one-form $\xi_{\bar{a}}$.
			Then the Killing connection of $(M \times \bM,g+\bg) $ is not exact.

To see this, and here
we
use the previous index conventions and follow \cite[Proposition~4]{CostanzaEastwoodLeistner21}, let $\omega_{AB}=\omega_{ab}\in \Gamma\left( \Wedge^{2,0}\right)$ be the K\"{a}hler form on $M$ with local potential $\phi_B=\phi_b\in \Gamma\left( \Wedge^{1,0}\right)$, i.e.~$\nabla_{[A}\phi_{B]}=\omega_{AB}$, and $\xi_A=\xi_{\bar a}\in \Gamma\left( \Wedge^{0,1}\right))$ be the parallel vector field on $\bM$. We set $h_{AB}=\phi_{(A}\xi_{B)}$,
\[\psi_{BCD}=2\nabla_{[C}h_{D]B}
=
\omega_{CD}\xi_B +\tfrac{1}{2}\left( \nabla_C\phi_B\xi_D-\nabla_D\phi_B\xi_C\right),\]
and $\eta=\begin{bmatrix}
h_{BC}
\\
\psi_{BCD}
\end{bmatrix}$. Then $\psi_{[BCD]}=0$, and for the exterior covariant derivative of $\eta$ it is
\[
\D^{\wedge}_{A}
\begin{bmatrix}
h_{B{C}}
\\
\psi_{BCD}
\end{bmatrix}
=
\begin{bmatrix}
0
\\
\nabla_{[A}\psi_{B]CD}- R_{CD}{}^{E}{}_{[A} h_{B]E}
\end{bmatrix}.
\]
Since $\omega_{AB}$ is parallel, it is $\nabla_A\nabla_B\phi_C=\nabla_A\nabla_C\phi_B$, so that, together with $\nabla_A\xi_B=0$,
\[
\nabla_{[A}\psi_{B]CD}
=
\tfrac{1}{2}\left( \nabla_{[A}\nabla_{B]}\phi_C \xi_D
-
\nabla_{[A}\nabla_{B]}\phi_D \xi_C\right)
=
-\tfrac{1}{2}
R_{AB}{}^E{}_{[C} \xi_{D]} \phi_E.
\]
On the other hand,
\[
R_{CD}{}^{E}{}_{[A} h_{B]E}
=
\tfrac{1}{2}R_{CD}{}^{E}{}_{[A} \xi_{B]}\phi_{E},
\]
so that,  for $\mu_{AB}=\phi_{[A}\xi_{B]}$,
\[
\D^{\wedge}_{A}
\begin{bmatrix}
h_{B{C}}
\\
\psi_{BCD}
\end{bmatrix}
=
\begin{bmatrix}
0
\\
R_{AB}{}^{E}{}_{[C} \mu_{D]E}+R_{CD}{}^{E}{}_{[A} \mu_{B]E}
\end{bmatrix}
\]
is in the range of the curvature of the Killing connection.

However, if $\eta$
 was in the range of the Killing connection, then there would be a one-form $\sigma_C=\sigma_c+\sigma_{\bar c}$
such that $h_{BC}=\nabla_B\sigma_C-\mu_{BC}$, and this, by the definition of $h_{BC}$ and $\mu_{BC}$, implies that
$\nabla_{B}\sigma_{C}=\phi_B \xi_C=\phi_b\xi_{\bar c}$.
Hence,
\[0= \nabla_b\sigma_c =\nabla_{\bar b}\sigma_c=\nabla_{\bar b}\sigma_{\bar c}\quad\text{ and }\quad
\nabla_{ b}\sigma_{\bar c}= \phi_b\xi_{\bar c}.\]
Therefore, $\sigma_c\in \Gamma\left( \Wedge^{1,0}\right)$ is a lift of a parallel vector field on $M$, and, with $(M,g)$ being indecomposable and Rieman\-nian, must be zero. The last equation implies that
\[0= \nabla_{[a}\nabla_{b]}\sigma_{\bar c} =\omega_{ab}\xi_{\bar c},\]
which is a contradiction, as the K\"{a}hler form and the parallel vector field are both not zero.
		\end{example}

As a result that holds in any signature, we can now easily see the following.
\begin{corollary}\label{cscproducts}
Let $(M,g)$ be a semi-Rieman\-nian manifold that is a product of two spaces of constant sectional curvature. Then the Killing connection is exact.
\end{corollary}
\begin{proof}
The Killling connection of flat space is exact, so we can assume that at least one of the factors has non-vanishing sectional curvature.
For spaces of constant sectional curvature, the Killling  connection is flat, so that $E=E_0$. Hence, all assumptions of Proposition \ref{prop:small_exactness} are satisfied, and we obtain the desired result.
\end{proof}

\section{Exactness of the Killing connection for Lorentzian locally symmetric
spaces}
\label{lorsec}

In this section we will provide the 
 {\em proof of
Theorem~\ref{calabi-theo}} for Lorentzian
locally symmetric spaces, using the results of the previous sections. For this we first will show that the Killing
connection for indecomposable Lorentzian symmetric spaces and for products of
such with a flat factor is exact.  Indecomposable Lorentzian locally symmetric
spaces either have constant sectional curvature or are locally isometric to a
Cahen--Wallach space~\cite{Cahen-Wallach70}.
Corollary \ref{cscproducts} gives
the result in the constant sectional curvature case, so we now focus on
Cahen--Wallach spaces.

In~\cite{Cahen-Wallach70} it was shown that an indecomposable Lorentzian symmetric space either has constant sectional curvature or is covered by a Cahen--Wallach space. These  are  special cases of plane waves and pp-waves. 
The latter were introduced in~\cite{brinkmann25} in the context of conformal geometry and play an important role in general relativity. In order to define Cahen--Wallach spaces,
consider the  Lorentzian manifold $(\R^{n+2},g)$ with
\begin{equation}
\label{CWmetric}
g=2\d x^-\d x^+ +(x^iQ_{ij}x^j)(\d x^+)^2+ \delta_{ij}\d x^i \d x^j,\end{equation}
where $(x^-,x^1, \ldots, x^n,x^+)$ are co\"ordinates on $\R^{n+2}$ and $Q_{ij}$ is a symmetric $(n\times n)$ matrix (of constants).
The Lorentzian manifold $(\R^{n+2},g)$ has a parallel null vector field $\partial_-$, is symmetric, and its curvature is given by
\[R_{+ij+}=Q_{ij},\]
and all other terms not determined by the symmetries of $R$ being zero.
The isometry classes of metrics as in (\ref{CWmetric}) are determined by the eigenvalues and eigenspaces of the matrix $Q_{ij}$. If $Q_{ij}$ is non-degenerate, then the metric $g$ is indecomposable and $(\R^{n+2},g)$ is called a {\em Cahen--Wallach space}. If $Q_{ij}$ is degenerate, then $(\R^{n+2},g)$ is isometric to a product of a Cahen--Wallach space of dimension $\mathrm{rk}(Q_{ij})+2$ and Euclidean space.

In the following we do not assume that $Q_{ij}$ is non degenerate and we let $1\le q\le n$ be its rank.
We  fix a frame  $(e_-,e_1,\ldots, e_n, e_+)$, 			\[
				e_{-} = \partial_{-}, \quad e_{i} = \partial_{i}, \quad e_{+} = - \tfrac{1}{2}(x^iQ_{ij}x^j)  \partial_{-} + \partial_{+},
				\]
so that 
 $\ker(Q_{ij}) = \mathrm{span} \lbrace e_{q+1}, \dots , e_{n} \rbrace$,
			and write $\R^{q} = \mathrm{span} \lbrace e_{1}, \dots, e_{q} \rbrace$.
			In this frame, an endomorphism $A \in \lie{so}(TM, g) \simeq \lie{so}(1, n+1)$ takes the form
			\begin{equation}\label{eqn:so_matrix}
				A =
				\begin{pmatrix}
					a & u^{t} & 0 \\
					v & B & - u \\
					0 & - v^{t} & -a
				\end{pmatrix},
				\quad a \in \R, \; u, v \in \R^{n} \; \text{and} \; B \in \lie{so}(n).
			\end{equation}
Since $g$ is locally symmetric, the holonomy algebra is given by the curvature,
so that
			\begin{equation}\label{holonomy_matrix}
				\hol = \{ R(w,e_+)\mid w\in \R^n\}
				=
\left\{				\begin{pmatrix}
					0 & u^{t} & 0 \\
					0 & 0 & - u \\
					0 & 0 & 0
				\end{pmatrix}
\mid u \in \R^{q}\right\}.
			\end{equation}

			\begin{theorem}\label{theo:CW_exactness}
				Let $(M, g)$ be a Cahen--Wallach space or a product of Cahen--Wallach space with Euclidean space. Then the curvature filtration for the Killing connection on $M$ satisfies $E = E_{2}$. In particular, the Killing connection $\D$ of $(M, g)$ is exact.
			\end{theorem}
			\begin{proof} In order to show that $E = E_{2}$, we  compute explicitly the Lie algebras $\lie{h}_{k}$, $k = 0,1,2$, defined in (\ref{hkdef}).

			The first case to consider is $\lie{h}_{0}$. The Lie algebra $\lie{h}_{0}$ is isomorphic to the algebra of automorphisms of the curvature $\aut(R)$, which, as it a was shown in 			  \cite[Section 4.3.4]{GlobkeLeistner16}), is isomorphic to
\begin{equation}\label{eqn:h0CW}
				\h_0=Z_{\lie{so}(n)}(Q) \ltimes \R^{n}
				=
				\left\{
							A =
				\begin{pmatrix}
				0 & z^{t} & 0 \\
				0 & B & -z \\
				0 & 0 & 0 \\
				\end{pmatrix}\mid
 z \in \R^{n},\, B \in Z_{\lie{so}(n)}(Q) \right\}.
			\end{equation}			 Here $Z_{\lie{so}(n)}(Q)$ is the centraliser in $\so(n)$ of the matrix $Q$.
In fact, $Z_{\lie{so}(n)}(Q) = Z_{\lie{so}(q)}(Q) \oplus \lie{so}(n-q)	$.
						For the sake of completeness, we will show this by analysing the constraints that an automorphism of the curvature tensor must satisfy. Pick  an element $A$ of $\so(TM,g)$ as in (\ref{eqn:so_matrix}) and let $A$ act on $R$. We show that if $A$ as in (\ref{eqn:so_matrix}) is an element in $\h_0$, then  $v = 0$, $a=0$ and $B\in Z_{\lie{so}(n)}(Q) $, where we use that $e_-\hook R=e_i\hook R=0$,  for $e_i\in \ker (Q)$, and $R(e_i,e_j)=0$, for all $1\le i,j\le n$. Recalling (\ref{action}), the $e_{-} \wedge e_{+}$ component of $A \cdot R$ yields
			\[
				0=(A \cdot R)(e_{-}, e_{+})  = - R(v, e_{+}) = {\begin{pmatrix}
					0 & v^{t}Q & 0 \\
					0 & 0 & - Qv \\
					0 & 0 & 0
				\end{pmatrix}},
				\]
so that $v \in \ker(Q)$.  Now for every  $y \in \R^{q}$,
			\[
				0 = (A \cdot R)(v, y) = g(v, v) R(e_{+}, y),
				\]
so that we   conclude that $v = 0$.
Next, for $x \in \ker(Q)$, it follows that
			\[
				0 = (A \cdot R)(x, e_{+}) =- R(Bx, e_{+}),
				\]
				which shows that $B$ leaves the kernel of $Q$ invariant and, since $B\in \so(n)$, also its complement $\R^q$.

For $y \in \R^{q}$, it follows that
			\begin{eqnarray*}
				0 = (A \cdot R)(y, e_{+}) &=& [A, R(y, e_{+})]- R(Ay, e_{+}) - R(y, Ae_{+}) \\
				&=&
	[A, R(y, e_{+}) ]	- R(By, e_{+}) +a R(y, e_{+}).
				\end{eqnarray*}
The equation that results from computing the commutator yields
\[ [Q,B]y-2a Qy=0,\]
for all $y\in \R^q$, which implies the matrix equation $[Q,B]=2a Q$. Using 		the trace-form to pair this equation with $Q$, and given that $Q$ is symmetric, this yields $a=0$ and therefore $[Q,B]=0$. This proves statement (\ref{eqn:h0CW}).

			Now, we will show that $\lie{h}_{1} $ is equal to the stabiliser  in $\so(1,n+1)$ of the null line $\R\cdot e_-$, i.e.~that
\begin{equation}
			\label{h1}
			\lie{h}_{1} = \lie{co}(n) \ltimes \R^{n}=(\R\oplus \lie{so}(n)) \ltimes \R^{n}.
			\end{equation}
Recall that $			\lie{h}_{1} = \lbrace A \in \lie{so}(1,n+1) \mid [A, H] \in \lie{h}_{0}, \forall H \in \hol \rbrace$.  Taking $A \in \lie{so}(1,n+1)$ and $H \in \hol$ as in (\ref{eqn:so_matrix}) and (\ref{holonomy_matrix}) respectively, it follows that their Lie bracket is
			\begin{equation}\label{eqn:bracket_so_hol}
				[A, H] =
				\begin{pmatrix}
				-v^{t}w & (aw + Bw)^{t} & 0 \\
				0 & vw^{t} - wv^{t} & -(aw + Bw) \\
				0 & 0 & v^{t}w
				\end{pmatrix}.
			\end{equation}
			The first thing to notice is that, for $[A, H]$ to be in $\h_0$, ie.~as in (\ref{eqn:h0CW}), there are no constraints on $a$, $u$ and $B$, whereas for $v$ the equations
			\[
				v^{t}w = 0 \quad \text{and} \quad vw^{t} - wv^{t} \in Z_{\lie{so}(n)}(Q)
				\]
			must be satisfied for all $w \in \R^{q}$. The first equation implies that $v \in \ker(Q)$, since $w \in \R^{q}$, and therefore $vw^{t} - wv^{t}$ takes an off diagonal form. The centraliser of $Q$ in $\lie{so}(n)$ is block diagonal, hence $v$ is in fact equal to $0$. Therefore $\lie{h}_{1}$ is given by
			\[
				\lie{h}_{1} =
				\left\lbrace
					\begin{pmatrix}
					a & u^{t} & 0 \\
					0 & B & - u \\
					0 & 0 & -a
					\end{pmatrix}
					\; : \;
					a \in \R, \; u \in \R^{n}, \; B \in \lie{so}(n)
				\right\rbrace,
				\]
			that is, $\lie{h}_{1}$ is isomorphic to the stabiliser $\lie{co}(n) \ltimes \R^{n}$ of the null line $\R e_{-}$, as claimed.

			Lastly, from equation (\ref{eqn:bracket_so_hol}), follows that \[\lie{h}_{2} = \lbrace A \in \lie{so}(1,n+1) \mid [A, H] \in \lie{h}_{1}, \forall H \in \hol\rbrace = \lie{so}(1,n+1),\] and thus $E = E_{2}$.

			Finally, from Proposition \ref{Ekparallel} we know that $E_0$ and $E_1$  are parallel. Therefore, from Proposition \ref{EeqEk} we obtain that the Killing connection of $(M,g)$ is exact.
			\end{proof}

Now we are ready to prove the main result of this section, which, by  Corollary~\ref{2exact}, will imply Theorem~\ref{calabi-theo}.
\begin{theorem}\label{killcon-exact}
Let $(M,g)$ be a Lorentzian locally symmetric space.
Then the Killing connection is  exact unless the de~Rham decomposition of $(M,g)$ contains a Hermitian factor and a factor that is flat or  a Cahen--Wallach space, in which case the Killing connection is not exact.
\end{theorem}

\begin{proof}
		Let $M^{1} \times \dots \times M^{k} \times L$ be the local de~Rham decompositon of $M$ into irreducible Rieman\-nian factors $M^i$ and a Lorentzian factor $L$, such that $L$ does not contain a non-flat Rieman\-nian factor, that is, $L$ is either
		\begin{enumerate}
		\item indecomposable Lorentzian, i.e.~with non-zero constant sectional curvature or
a Cahen--Wallach space, or
		\item  a product of an indecomposable Lorentzian symmetric space with a Euclidean factor, or
		\item   Minkowski space.
		\end{enumerate} In all three cases, Corollary \ref{cscproducts} and Theorem \ref{theo:CW_exactness} imply that the Killing connection of $L$ is exact.
Moreover, by  \cite{CostanzaEastwoodLeistner21}
the same holds for $M^{k}$ (see also Remark \ref{riemremark}), so that we can apply Proposition~\ref{prop:small_exactness} to obtain that the Killing connection is exact for $M^{k} \times L $, provided that $M^{k}$ is an irreducible Rieman\-nian symmetric space that is non-Hermitian if $L$ admits a parallel vector field. Inductively, it follows that the Killing connection on $M$ is exact unless it contains a Hermitian factor in its local de~Rham decomposition and $L$ admits a parallel vector field.
\end{proof}

\bibliographystyle{abbrv}
\bibliography{GEOBIB}
\end{document}